\documentclass[12pt,a4paper,reqno]{amsart}
\usepackage[top=1.15in, bottom=1.15in, left=1.17in, right=1.17in]{geometry}
\usepackage[colorlinks=true, pdfstartview=FitV, linkcolor=blue, citecolor=blue, urlcolor=blue, breaklinks=true]{hyperref}

\usepackage{amssymb}
\usepackage{mathrsfs}
\usepackage{todonotes}
\usepackage{amscd}
\usepackage{comment}
\usepackage{dsfont}
\usepackage{fancyhdr, color,tikz-cd}
\usepackage[backend=bibtex,style=alphabetic,doi=true,isbn=false,url=false,giveninits=true,maxbibnames=50]{biblatex}
\newcommand{\bd}{\mathbf{d}}

\newcommand{\be}{\mathbf{e}}
\newcommand{\bc}{\mathbf{c}}

\newcommand{\g}{\mathfrak{g}}

\newcommand{\ve}{\varepsilon}
\newcommand{\vp}{\varphi}

\newcommand{\wt}{\widetilde}

\DeclareMathOperator{\supp}{supp}
\DeclareMathOperator{\im}{Im}

\newcommand{\n}{\mathfrak{n}}
\newcommand{\h}{\mathfrak{h}}

\numberwithin{equation}{section}

\newtheorem{theorem}{Theorem}[section]
\newtheorem*{thm*}{Theorem}
\newtheorem{corollary}[theorem]{Corollary}
\newtheorem{lemma}[theorem]{Lemma}
\newtheorem{proposition}[theorem]{Proposition}
\newtheorem{conjecture}[theorem]{Conjecture}

\newtheorem{definition}[theorem]{Definition}
\newtheorem{remark}[theorem]{Remark}

\title{Dynkin abelianisations of flag varieties}

\author[Enugandla, Fang, Fourier and Steinert]{Shreepranav Varma Enugandla, Xin Fang, Ghislain Fourier and Christian Steinert }
\address{Shreepranav Varma Enugandla:\newline
University of California, Berkeley, United States of America}
\email{shreepranav@gmail.com}
\address{Xin Fang:\newline
RWTH Aachen University, Chair of Algebra and Representation Theory, Aachen, Germany}
\email{xinfang.math@gmail.com}
\address{Ghislain Fourier:\newline
RWTH Aachen University, Chair of Algebra and Representation Theory, Aachen, Germany}
\email{fourier@art.rwth-aachen.de}
\address{Christian Steinert:\newline
RWTH Aachen University, Chair of Algebra and Representation Theory, Aachen, Germany}
\email{steinert@art.rwth-aachen.de}

\begin{document}

\begin{abstract}
Cerulli Irelli and Lanini have shown that PBW degenerations of flag varieties in type $\tt A$ and $\tt C$ are actually Schubert varieties of higher rank. We introduce Dynkin cones to parameterise specific abelianisations of classical Lie algebras. Within this framework, we generalise their result to all degenerations of flag varieties defined by degree vectors originating from a Dynkin cone. This framework allows us to determine the extent to which a flag variety can be degenerate while still naturally being a Schubert variety of the same Lie type. Furthermore, we compute the defining relations for the corresponding degenerate simple modules in all classical types.
\end{abstract}
\maketitle
\section{Introduction}

Let $\g$ be a finite-dimensional complex simple Lie algebra of classical type with a fixed triangular decomposition $\g = \n_+ \oplus \h \oplus \n_-$. The connected simply-connected semi-simple algebraic group associated with $\g$ is denoted by $G$. For a parabolic subgroup $P$ of $G$, the projective variety $G/P$ is called a \emph{partial flag variety}. The classical (type $\tt A$) flag variety occurs in the special case where $G=\mathrm{SL}_n(\mathbb{C})$ and $P$ is the standard Borel subgroup of $G$.

The natural Poincaré-Birkhoff-Witt (PBW) filtration on the universal enveloping algebra of a Lie algebra provides an algebra filtration on $U(\n_-)$ and, by extension, a filtration on any finite-dimensional simple $\g$-module $V$ when a highest weight generator is fixed. The induced associated graded module $V^a$ becomes a cyclic module for the abelian Lie algebra $\n_-^a$ over the underlying vector space $\n_-$. The PBW filtration on $U(\n_-)$ is stable under the actions of $\mathfrak{b}:=\n_+\oplus \h$, and this stability leads to an action of $\g^a:=\mathfrak{b} \ltimes \n_-^a$ on $V^a$. These PBW degenerate modules have been studied in \cite{FFL11a, FFL11b} for $\g$ of type $\tt A$ and $\tt C$.

Consequently, \cite{Fei12} translated this degeneration in type $\tt A$ into a geometric context, by defining the PBW degenerate flag variety as the closure of the action of the algebraic group $G^a$ corresponding to $\g^a$ on the highest weight line in $\mathbb{P}(V)$. The PBW degenerate flag variety is a flat degeneration of $\mathrm{SL}_n/B$. In subsequent paper \cite{Fei11}, a realisation of type $\tt A$ in terms of quiver Grassmannians was given. Here, the defining quiver representation comprises a sequence of projections, in contrast to the identity maps used for the classical flag variety. Similarly, \cite{FFL14} considered the PBW degenerate symplectic flag variety.

In type $\tt A$, this realisation in terms of quiver Grassmannians, as used implicitly in \cite{CIL15}, demonstrates that the PBW degenerate flag variety is in fact a Schubert variety within a partial flag variety of the same Lie type but of higher rank. Similar results in type $\tt C$ are obtained using a folding procedure. They explicitly provided the Weyl group element defining the Schubert variety. In \cite{CLL}, similar results were rediscovered using a different construction of embedding $\g^a$ into matrices of larger sizes. It is shown that the PBW degenerate modules are Demazure modules of a Lie algebra of higher rank, and their defining relations can be naturally obtained from those of the Demazure modules.

The purpose of this project is to generalise the above results in aiming to determine, for all Lie algebras of classical type, the extent to which one can degenerate a flag variety and still naturally obtain a Schubert variety of the same Lie type (but of larger rank). We introduce a polyhedral cone $\mathcal{D}(\mathfrak g)$, describing the degrees one can associate with negative roots, resulting in an associated graded Lie algebra that is a partial abelianisation of $\n_-$. For $\bd\in\mathcal{D}(\g)$ we let $\n_-^\bd$ denote this partially abelianised Lie algebra. Some of such abelianisations of $\n_-$ are uniquely determined by fixing a set of pairs of laced simple roots and ensuring that their root vectors commute. One can interpret this as effectively stretching the Dynkin diagram and inserting a new node between each pair of laced simple roots; hence, we call them \emph{Dynkin abelianisations}. Denoting by $\mathbf{c}$ such a set of pairs, we introduce the corresponding \emph{Dynkin cone} $F^{\mathbf{c}} \subseteq \mathcal{D}(\mathfrak g)$, which enforces the required property while also incorporating certain additional equalities, which we call the differential operators equalities. For $\mathbf{d}$ in the relative interior of the Dynkin cone $F^{\mathbf{c}}$, we denote the associated graded module of a finite-dimensional simple $\g$-module by $V^{\mathbf{d}}$.

The main result of the paper is the following, see Theorem \ref{Thm:IsoModule} for a precise statement.

\begin{thm*}
Let $\g$ be of classical type $\tt X_n$, $\wt{\g}$ of type $\tt X_{n+ \#\mathbf{c}}$, and $V$ a finite-dimensional simple $\g$-module. There exists an element $w$ in the Weyl group of $\wt\g$ such that $V^{\mathbf{d}}$ is isomorphic to a Demazure module of $\wt\g$ associated to $w$ as modules for $\n_-^\bd$.
\end{thm*}

The $\bd$-degenerate flag variety is defined similarly to the degenerate flag variety as a highest weight orbit in $\mathbb{P}(V^\bd)$ of the action of an algebraic group associated to $\n_-^\bd$. Consequently, we proved that

\begin{thm*}[Corollary \ref{Cor:IsoProj}]
For $\mathbf{d}$ in the relative interior of $F^{\mathbf{c}}$, the corresponding $\bd$-degenerate flag variety is isomorphic to a Schubert variety in a partial flag variety of higher rank.
\end{thm*}

A few comments should be made here about the proof. We are following the strategy of \cite{CLL}, using the embedding trick. We initially prove the isomorphism for the modules corresponding to fundamental weights. Then, we utilise the fact that any simple, finite-dimensional module can be embedded in a tensor product of fundamental modules. This embedding leads to a surjective map from the degenerate module to the Demazure module. Using the defining ideal for Demazure modules \cite{Jos}, we then show that there is a surjective map from the Demazure module to the degenerate module, thus establishing the necessary isomorphism.

Our proof yields further insights beyond the initial results. A key challenge lies in describing the defining ideal of $V^{\mathbf{d}}$. This challenge has been addressed for the PBW degenerate module in types $\tt A$ and $\tt C$ in \cite{FFL11a, FFL11b}, and for type $\tt A$ generally in \cite{CFFFR17}. With our identification with the Demazure module, we conclude for all classical types:

\begin{thm*}
The defining ideal of $V^{\mathbf{d}}$ is generated by $\mathfrak{b}$ and the set $U(\n_+)\cdot f_{\alpha}^{\ell_\alpha}$, where $f_\alpha$ are images of the root vectors in $\n^{\mathbf{d}}$ and $\ell_\alpha$ are appropriate powers.
\end{thm*}

This concept may be familiar to the reader, as the defining relations of the finite-dimensional  simple $\g$-modules are similar.

From identifying the $\bd$-degenerate flag variety with the Schubert varieties, we deduce that the $\bd$-degenerate varieties are normal, Cohen-Macaulay, and have rational singularities. We conjecture that this statement holds even if $\mathbf{d}$ is not in the relative interior, although our current methods do not extend to this scenario.

Describing monomial bases of the PBW degenerate modules is quite challenging, and this task has been accomplished for types $\tt A, C$ (\cite{FFL11a, FFL11b}), the results are known as FFLV-bases, and for type $\tt B$ for almost abelian degenerations (\cite{Makh}). The combinatorial aspects of our results, along with those by \cite{CLL}, relate these bases in types $\tt A$ and $\tt C$ to the so-called string bases (\cite{Lit98, BZ01}) as shown explicitly in \cite{CFL24}. One future direction of our work is to adapt the string bases of the Demazure modules in type $\tt D$ into FFLV-type bases.

Due to their complexity, we are restricting ourselves to classical types here. The exceptional types have been discussed to some extent by the first author in \cite{Enu22}. Our construction, which involves degenerations by stretching the Dynkin diagram, can also extend to cases where the stretched diagram is not of finite type. In these instances, one obtains Demazure modules for affine types, though the combinatorics of the proofs still need to be worked out. Even beyond this, if the resulting stretched diagram extends beyond affine type, one must then explore Demazure modules for hyperbolic Kac-Moody algebras.

The paper is structured as follows: in Section~\ref{sec:abelian}, we describe the cone of abelianisations; in Section~\ref{sec:dynkin}, we introduce the Dynkin cone and state the main result. Section~\ref{sec:reduction} and Section~\ref{Sec:LemmaProof} reduce the proof of the main theorem to fundamental modules, which are then analysed in detail in Section~\ref{Sec:Fund}.
\vskip 10pt
\noindent
\textbf{Acknowledgements:} The work of GF is funded by the Deutsche
Forschungsgemeinschaft (DFG, German Research Foundation): “Symbolic Tools in
Mathematics and their Application” (TRR 195, project-ID 286237555).

\section{Partial abelianisations of Lie algebras and representations}\label{sec:abelian}

Throughout the entire paper, $[n]=\{1,2,\ldots,n\}$.

Let $\g$ be a finite dimensional complex simple Lie algebra of rank $n$ and fix a triangular decomposition $\g = \n_+ \oplus \h \oplus \n_-$ of $\g$. Let $\Phi$ be the root system of $\g$, $\Phi^+$ the set of positive roots in $\Phi$ and $\{\alpha_1,\ldots,\alpha_n\}$ the set of simple roots in $\Phi^+$. The coroot of a root $\beta\in\Phi$ is defined as $\beta^\vee:=2\beta/(\beta,\beta)$. The Weyl group is denoted by $W$ and the simple reflections $\{ s_1, \ldots, s_n\}$. Let $\Lambda^+$ denote the set of dominant integral weights and $\varpi_1,\ldots,\varpi_n$ denote the fundamental weights. 

We fix a Chevalley basis of $\g$ \cite[Section 25]{Hum}, which is a weighted basis 
$$\{e_\beta,f_\beta,h_i\mid \beta\in\Phi^+,\ \ 1\leq i\leq n\}$$
where $e_\beta$ (resp. $f_\beta$) is of weight $\beta$ (resp. $-\beta$). Results in this section do not depend on the choice of a weight basis of $\g$. Nevertheless, for the main results, the existence of a Chevalley basis is used in the proofs.

For $\lambda \in \Lambda^+$, we denote by $V(\lambda)$ the finite-dimensional simple $\g$-module of highest weight $\lambda$, $v_\lambda$ denotes a generator of the highest weight space. Further, for $w \in W$, we denote the Demazure module $V_w(\lambda):= U(\mathfrak{b})\cdot v_{w(\lambda)}$, where $v_{w(\lambda)} \in V(\lambda)$ is a generator of the one-dimensional weight space of weight $w(\lambda)$.

For a positive root $\beta$, we define its support by 
$$\mathrm{supp}(\beta):=\{i\mid i\in[n]\text{ with }(\varpi_i,\beta^\vee)\neq 0\}.$$

\subsection{Cone of partial abelianisations}

We define a polyhedral cone $\mathcal{D}(\mathfrak{g})\subseteq\mathbb{R}^{\Phi^+}$ by the following inequalities: a point $\mathbf{d}=(d_\beta)_{\beta\in\Phi^+}\in\mathbb{R}^{\Phi^+}$ is contained in $\mathcal{D}(\mathfrak{g})$, if for any pair of positive roots $\beta_1,\beta_2\in\Phi^+$ such that $\beta_1+\beta_2\in\Phi^+$, 
\begin{equation}\label{Eq:Cone}
d_{\beta_1}+d_{\beta_2}\geq d_{\beta_1+\beta_2}.
\end{equation}

The cone $\mathcal{D}(\g)$ will be called the \emph{cone of partial abelianisations} of $\g$, for reasons which will become clear later. The following proposition provides polyhedral geometric properties of the cone $\mathcal{D}(\g)$.

\begin{proposition}\label{Prop:Cone}
\begin{enumerate}
\item The polyhedral cone $\mathcal{D}(\mathfrak{g})$ is full dimensional in $\mathbb{R}^{\Phi^+}$.
\item The defining inequalities in \eqref{Eq:Cone} are non-redundant, that is to say, they give the facet description of $\mathcal{D}(\mathfrak{g})$.
\end{enumerate}
\end{proposition}

\begin{proof}
\begin{enumerate}
\item The function $\mathbf{2}$ taking value $2$ at all positive roots satisfies the strict inequalities in \eqref{Eq:Cone}, hence the cone is of full dimension.
\item We show that if an inequality $d_{\alpha}+d_{\beta}\geq d_{\gamma}$ for $\alpha,\beta,\gamma\in\Phi^+$ with $\alpha+\beta=\gamma$ is missing, there exists a point which is not contained in $\mathcal{D}(\mathfrak{g})$. For this, consider a point $\mathbf{d}\in\mathbb{R}^{\Phi^+}$ defined by:
$$\mathbf{d}:=\mathbf{2}-\mathbf{e}_\alpha-\mathbf{e}_\beta+\mathbf{e}_\gamma,$$
where $\mathbf{e}_\delta$ is the coordinate function of $\delta\in\Phi^+$. 

This point $\mathbf{d}$ fulfills all inequalities in \eqref{Eq:Cone} except $d_{\alpha}+d_{\beta}\geq d_{\gamma}$. Indeed, $d_\alpha=d_\beta=1$ but $d_\gamma=3$ so the above inequality is violated. We look at all other inequalities in \eqref{Eq:Cone}. 
\begin{itemize}
\item[-] If $\#\{\beta_1,\beta_2\}\cap\{\alpha,\beta\}=2$ then we are in the above case.
\item[-] If $\#\{\beta_1,\beta_2\}\cap\{\alpha,\beta\}=1$, then in the inequality \eqref{Eq:Cone} the left hand side is larger or equal to $3$, but all coordinate of $\mathbf{d}$ are less or equal to $3$. 
\item[-] If $\#\{\beta_1,\beta_2\}\cap\{\alpha,\beta\}=0$, then in the inequality \eqref{Eq:Cone} the left hand side is larger or equal to $4$, but all coordinate of $\mathbf{d}$ are less or equal to $3$. \end{itemize}
\end{enumerate}
\end{proof}

The polyhedral cone $\mathcal{D}(\mathfrak{g})$ has an $n$-dimensional lineality space. We will not need this result so the proof is omitted.

\subsection{Filtrations on enveloping algebras}

Recall that for a fixed $\bd\in\mathcal{D}(\g)$ and $\beta\in\Phi^+$, $d_\beta:=\bd(\beta)$. We interpret the cone $\mathcal{D}(\g)$ using filtration on Lie algebras.

For $m\in\mathbb{R}$, we define the following subspaces of $\mathfrak{n}_-$ and $U(\n_-)$:
$$(\n_{-})^{\mathbf{d}}_{\leq m}=\mathrm{span}\{f_\beta\mid\, \beta\in\Phi^+,\  d_\beta\leq m\},$$
$$U(\n_{-})^{\mathbf{d}}_{\leq m}=\mathrm{span}\{f_{\beta_1}\cdots f_{\beta_\ell}\mid  d_{\beta_1}+\ldots+d_{\beta_\ell}\leq m \}.$$

The following proposition follows from \eqref{Eq:Cone}.

\begin{proposition}
For a fixed $\mathbf{d}\in\mathcal{D}(\g)$,
\begin{enumerate}
\item the subspaces $\{(\n_{-})^{\mathbf{d}}_{\leq m}\mid m\in\mathbb{R}\}$ define an $\mathbb{R}$-filtration of Lie algebra on $\n_-$;
\item the subspaces $\{U(\n_{-})^{\mathbf{d}}_{\leq m}\mid m\in\mathbb{R}\}$ define an $\mathbb{R}$-filtration of algebra on $U(\n_-)$.
\end{enumerate} 
\end{proposition}

The subspaces 
$$(\n_{-})^{\mathbf{d}}_{< m}\subseteq (\n_{-})^{\mathbf{d}}_{\leq m}\ \ \text{and}\ \ U(\n_{-})^{\mathbf{d}}_{<m}\subseteq U(\n_{-})^{\mathbf{d}}_{\leq m}$$ 
are defined in replacing the inequalities in the definitions of $(\n_{-})^{\mathbf{d}}_{\leq m}$ and $U(\n_{-})^{\mathbf{d}}_{\leq m}$ by the strict ones. The associated graded Lie algebra and the associated graded algebra are defined by:
$$\n_-^{\mathbf{d}}:=\bigoplus_{m\in\mathbb{R}}(\n_-)^{\mathbf{d}}_{\leq m}/(\n_-)^{\mathbf{d}}_{<m}\ \ \text{and}\ \ U(\n_-)^{\mathbf{d}}:=\bigoplus_{m\in\mathbb{R}}U(\n_-)^{\mathbf{d}}_{\leq m}/U(\n_-)^{\mathbf{d}}_{<m}.$$

For $\beta\in\Phi^+$, the class of $f_\beta$ in $\mathfrak{n}_-^{\mathbf{d}}$ will be denoted by  $f_\beta^{\mathbf{d}}$.

According to the following proposition, taking associated graded and taking the universal enveloping commute. In the following we will not distinguish them.

\begin{proposition}\label{Prop:Functorial}
The linear map $\mathfrak{n}_-^{\mathbf{d}}\to U(\mathfrak{n}_-)^{\mathbf{d}}$ sending $f_\beta^{\mathbf{d}}$ to the class of $f_\beta$ in $U(\mathfrak{n}_-)^{\mathbf{d}}$ induces an isomorphism of algebras $U(\n_-^{\mathbf{d}})\cong U(\mathfrak{n}_-)^{\mathbf{d}}$.
\end{proposition}

\begin{proof}
By \eqref{Eq:Cone} and the universal property of the enveloping algebra, there exists an algebra homomorphism $U(\mathfrak{n}_-^{\mathbf{d}})\to U(\mathfrak{n}_-)^{\mathbf{d}}$. The PBW theorem can be then applied to conclude.
\end{proof}

Let $\mathrm{relint}(C)$ denote the relative interior of a polyhedral cone $C$.

By the following lemma, the isomorphism type of the Lie algebra $\n_-^{\mathbf{d}}$, when $\bd$ varies in $\mathcal{D}(\g)$, depends only on the relative interior of each face. 

\begin{lemma}\label{Lem:Face}
Let $F$ be a face of $\mathcal{D}(\g)$ and $\mathbf{d},\mathbf{e}\in\mathrm{relint}(F)$.
\begin{enumerate}
\item The linear map 
$$\n_-^{\mathbf{d}}\to \n_-^{\mathbf{e}},\ \  f_\beta^{\mathbf{d}}\mapsto f_\beta^{\mathbf{e}}, \ \ \text{for}\ \beta\in\Phi^+,$$ 
is an isomorphism of Lie algebras.
\item The induced algebra homomorphism $U(\n_-)^{\mathbf{d}}\to U(\n_-)^{\mathbf{e}}$ is an isomorphism.
\end{enumerate}
\end{lemma}

\begin{proof}
The second part follows from the first part. To show the first part, it suffices to apply Proposition \ref{Prop:Cone}(2).
\end{proof}

When $\bd$ is taken from the (relative) interior of $\mathcal{D}(\g)$, we have:

\begin{lemma}\label{Lem:Abelian}
If $\mathbf{d}\in\mathrm{relint}(\mathcal{D}(\g))$, $U(\n_-)^{\mathbf{d}}$ is isomorphic as an algebra to the symmetric algebra $S(\n_-)$.
\end{lemma}

\subsection{Filtrations on modules}

Given $\bd\in\mathcal{D}(\g)$, the algebra filtration on $U(\n_-)$ introduced in the previous subsection induces a vector space filtration on cyclic modules. 

For any dominant integral weight $\lambda\in\Lambda^+$, the simple $\g$-module $V(\lambda)=U(\n_-)\cdot v_\lambda$ is cyclic. We consider the induced $\mathbb{R}$-filtration on $V(\lambda)$:
$$(V(\lambda))_{\leq m}^{\mathbf{d}}=U(\n_-)^{\mathbf{d}}_{\leq m}\cdot v_\lambda,$$
and similarly its subspace $(V(\lambda))_{< m}^{\mathbf{d}}$.
The associated graded space will be denoted by $V^{\mathbf{d}}(\lambda)$, \emph{i.e.}
$$V^{\mathbf{d}}(\lambda):=\bigoplus_{m\in \mathbb{R}}\ (V(\lambda))^{\mathbf{d}}_{\leq m}/(V(\lambda))^{\mathbf{d}}_{<m}.$$ 

The vector space $V^{\mathbf{d}}(\lambda)$ carries a graded $U(\n_{-}^{\mathbf{d}})$-module structure. Indeed, for any $k,\ell\in \mathbb{R}$, we have:
$$U(\n_-)_{\leq k}^\bd\cdot (V(\lambda))_{\leq \ell}^\bd\subseteq (V(\lambda))_{\leq k+\ell}^\bd.$$
Let $v_\lambda^{\mathbf{d}}$ denote the image of $v_\lambda$ in $V^{\mathbf{d}}(\lambda)$. The $U(\n_{-}^{\mathbf{d}})$-module $V^{\mathbf{d}}(\lambda)$ is cyclic, having $v_\lambda^{\mathbf{d}}$ as a cyclic vector.

The $U(\mathfrak{n}_-^{\mathbf{d}})$-module $V^{\mathbf{d}}(\lambda)$ will be called a \textit{weighted PBW degeneration} of $V(\lambda)$. Certain examples of weighted PBW degenerations are studied in type $\tt A$ and $\tt C$ \cite{FFFM, BF24}. It is shown in these works that such degenerations are closely related to the tropicalisation of flag varieties.

\subsection{Degenerate flag varieties}

We fix $\mathbf{d}\in\mathcal{D}(\g)$ and let $N_-^\bd:=\mathrm{exp}(\n_-^{\bd})$ denote the connected Lie group with Lie algebra $\n_-^{\bd}$. 

\begin{definition}
For a fixed $\lambda\in\Lambda^+$, the $\bd$-degenerate flag variety associated to $\lambda$ is defined by:
$$\mathcal{F}^{\bd}(\lambda):=\overline{N^{\bd}_-\cdot [v_\lambda^\bd]}\hookrightarrow \mathbb{P}(V^{\bd}(\lambda)).$$
\end{definition}

One of the goals of the paper is to study geometric properties of $\mathcal{F}^{\bd}(\lambda)$ for certain $\bd\in\mathcal{D}(\g)$.


\section{Dynkin abelianisations: main result}\label{sec:dynkin}

\subsection{Dynkin cones in \texorpdfstring{$\mathcal{D}(\g)$}{in the degeneration cone}}\label{Subsec:ABCD}

The goal of this paper is to study the geometry of $\mathcal{F}^{\bd}(\lambda)$ for $\bd$ in certain polyhedral cones which will be termed \emph{Dynkin} cones. The reason for the choice of such a name will become clear later.


We start from fixing notations for classical Lie algebras, following \cite{Bou}. 

\subsubsection{Type $\tt A_n$}
For $1\leq i\leq j\leq n$, we denote the positive root
$$\alpha_{i,j}:=\alpha_{i}+\ldots+\alpha_j\in\Phi^+.$$
We will also use the notation $\alpha_i=\ve_i-\ve_{i+1}$ for $1\leq i\leq n$. In this notation, for $1\leq i< j\leq n$, $\alpha_{i,j-1}=\ve_i-\ve_j$.

For $\bd\in\mathcal{D}(\g)$ and $1\leq i\leq j\leq n$, we will write $d_{i,j}:=d_{\alpha_{i,j}}$ for short.

For a fixed subset $\bc=\{c_1,\ldots,c_t\}\subseteq [n-1]$, we define a polyhedral cone $F^{\bc}$ in $\mathcal{D}(\g)$ by the following inequalities and equalities (PA), stands for partial abelianisation: 
\begin{enumerate}
\item[-] for $1\leq i\leq j<\ell\leq n$ with $j\in\bc$, $d_{i,j}+d_{j+1,\ell}\geq d_{i,\ell}$;
\item[-] all other defining inequalities of $\mathcal{D}(\g)$ take equality,
\end{enumerate}
and the following equalities (DO), stands for differential operators:
\begin{enumerate}
\item[-] for $1\leq i< j\le k <\ell\leq n$ , $d_{i,k}+d_{j,\ell} = d_{i,\ell} + d_{j,k}$.
\end{enumerate}

\subsubsection{Type $\tt B_n$}
The positive roots are:
\begin{enumerate}
\item[-] for $1 \leq i\leq j\leq n$, $\alpha_{i,j}:=\alpha_i+\ldots+\alpha_j$;
\item[-] for $1\leq i< j\leq n$, $\alpha_{i,\overline{j}}:=\alpha_{i,n}+\alpha_{j,n}$.
\end{enumerate}
Similarly we have: for $1\leq i< j\leq n$, $\alpha_{i,j-1}=\ve_i-\ve_j$, for $1\leq i\leq n$, $\alpha_{i,n}=\ve_i$ and for $1\leq i<j\leq n$, $\alpha_{i,\overline{j}}=\ve_i+\ve_j$.

For $\bd\in\mathcal{D}(\g)$, we will write $d_{i,j}:=d_{\alpha_{i,j}}$ and $d_{i,\overline{j}}:=d_{\alpha_{i,\overline{j}}}$. The following total order will be considered:
$$1<2<\ldots<n<\overline{n}<\ldots<\overline{2}.$$

For a fixed subset $\bc=\{c_1,\ldots,c_t\}\subseteq [n-1]$, we define a polyhedral cone $F^{\bc}$ in $\mathcal{D}(\g)$ by the following inequalities and equalities (PA):
\begin{enumerate}
\item[-] for $1\leq i\leq j<\ell\ <\overline{j+1}$ with $j\in\bc$, $d_{i,j}+d_{j+1,\ell}\geq d_{i,\ell}$;
{
\item[-] for $1\leq i<k<j+1\leq n$ with $j\in\bc$, $d_{i,j}+d_{k,\overline{j+1}}\geq d_{i,\overline{k}}$;
\item[-] for $1\leq k<i<j+1\leq n$ with $j\in\bc$, $d_{i,j}+d_{k,\overline{j+1}}\geq d_{k,\overline{i}}$;}
\item[-] all other defining inequalities of $\mathcal{D}(\g)$ take equality,
\end{enumerate}
and the following equalities (DO):
\begin{enumerate}
\item[-] for $1\leq i< j\le k <\ell\leq n$ , $d_{i,k}+d_{j,\ell} = d_{i,\ell} + d_{j,k}$;
\item[-] for $1\leq i< j\le k ,\ell\leq n$ , $d_{i,\overline k}+d_{j,\ell} = d_{i,\ell} + d_{j,\overline k}$;
\item[-] for $1\leq i< j< k <\ell\leq n$ , $d_{i,\overline j}+d_{k,\overline \ell} = d_{i,\overline k} + d_{j,\overline \ell} = d_{i,\overline \ell}+d_{j,\overline k}$.
\end{enumerate}

\subsubsection{Type $\tt C_n$}
The positive roots are:
\begin{enumerate}
\item[-] for $1 \leq i\leq j\leq n$, $\alpha_{i,j}:=\alpha_i+\ldots+\alpha_j$;
\item[-] for $1\leq i\leq j\leq n-1$, $\alpha_{i,\overline{j}}:=\alpha_{i,n}+\alpha_{j,n-1}$.
\end{enumerate}
For convenience we usually use the notation $\alpha_{i,\overline{n}}:=\alpha_{i,n}$ for $1\leq i\leq n$.  Similarly we have: for $1\leq i< j\leq n$, $\alpha_{i,j-1}=\ve_i-\ve_j$, for $1\leq i\leq n$, $\alpha_{i,\overline{i}}=2\ve_i$ and for $1\leq i<j\leq n$, $\ve_i+\ve_j=\alpha_{i,\overline{j}}$. 

For $\bd\in\mathcal{D}(\g)$, denote $d_{i,j}:=d_{\alpha_{i,j}}$ and $d_{i,\overline{j}}:=d_{\alpha_{i,\overline{j}}}$. The following total order will be considered:
$$1<2<\ldots<n<\overline{n-1}<\ldots<\overline{1}.$$

For a fixed subset $\bc=\{c_1,\ldots,c_t\}\subseteq [n-1]$, we define a polyhedral cone $F^{\bc}$ in $\mathcal{D}(\g)$ by the following inequalities and equalities (PA):
\begin{enumerate}
\item[-] for $1\leq i\leq j\leq n-1$, $j+1\leq \ell\leq \overline{j+1}$ with $j\in\bc$, $d_{i,j}+d_{j+1,\ell}\geq d_{i,\ell}$;
\item[-] for $1\leq i\leq \ell\leq j+1\leq n$ with $j\in\bc$, $d_{i,j}+d_{\ell,\overline{j+1}}\geq d_{i,\overline{\ell}}$;
{
\item[-] for $1\leq \ell\leq i\leq j+1\leq n$ with $j\in\bc$, $d_{i,j}+d_{\ell,\overline{j+1}}\geq d_{\ell,\overline{i}}$;}
\item[-] all other defining inequalities of $\mathcal{D}(\g)$ take equality,
\end{enumerate}
and the following equalities (DO):
\begin{enumerate}
\item[-] for $1\leq i< j\le k <\ell\leq n$ , $d_{i,k}+d_{j,\ell} = d_{i,\ell} + d_{j,k}$;
\item[-] for $1\leq i< j\le k ,\ell\leq n$ , $d_{i,\overline k}+d_{j,\ell} = d_{i,\ell} + d_{j,\overline k}$;
\item[-] for $1\leq i< j< k <\ell\leq n-1$ , $d_{i,\overline j}+d_{k,\overline \ell} = d_{i,\overline k} + d_{j,\overline \ell} = d_{i,\overline \ell}+d_{j,\overline k}$.
\end{enumerate}

\subsubsection{Type $\tt D_n$}
The positive roots are:
\begin{enumerate}
\item[-] for $1 \leq i\leq j\leq n-1$, $\alpha_{i,j}:=\alpha_i+\ldots+\alpha_j$;
\item[-] for $1\leq i\leq n-2$, $\alpha_{i,n}:=\alpha_i+\ldots+\alpha_{n-2}+\alpha_n$;
\item[-] for $1\leq i< j\leq n-1$, $\alpha_{i,\overline{j}}:=\alpha_i+\ldots+\alpha_{n-2}+\alpha_{n-1}+\alpha_n+\alpha_{n-2}+\ldots+\alpha_j$.
\end{enumerate}
Similarly we have: for $1\leq i< j\leq n$, $\alpha_{i,j-1}=\ve_i-\ve_j$, for $1\leq i\leq n-2$, $\alpha_{i,n}=\ve_i+\ve_n$, $\alpha_n=\ve_{n-1}+\ve_n$ and for $1\leq i<j<n$, $\ve_i+\ve_j=\alpha_{i,\overline{j}}$.

For $\bd\in\mathcal{D}(\g)$, denote $d_{i,j}:=d_{\alpha_{i,j}}$ and $d_{i,\overline{j}}:=d_{\alpha_{i,\overline{j}}}$. The following total order will be considered:
$$1<2<\ldots<n<\overline{n-1}<\ldots<\overline{2}.$$

For a fixed subset $\bc=\{c_1,\ldots,c_t\}\subseteq  [n-3]$, we define a polyhedral cone $F^{\bc}$ in $\mathcal{D}(\g)$ by the following inequalities and equalities (PA):
\begin{enumerate}
\item[-] for $1\leq i\leq j< \ell < \overline{j+1}$ with $j\in\bc$, $d_{i,j}+d_{j+1,\ell}\geq d_{i,\ell}$;
{
\item[-] for $1\leq i<k<j+1\leq n$ with $j\in\bc$, $d_{i,j}+d_{k,\overline{j+1}}\geq d_{i,\overline{k}}$;
\item[-] for $1\leq k<i<j+1\leq n$ with $j\in\bc$, $d_{i,j}+d_{k,\overline{j+1}}\geq d_{k,\overline{i}}$;}
\item[-] all other defining inequalities of $\mathcal{D}(\g)$ take equality,
\end{enumerate}
and the following equalities (DO):
\begin{enumerate}
\item[-] for $1\leq i< j\le k <\ell\leq n$, $d_{i,k}+d_{j,\ell} = d_{i,\ell} + d_{j,k}$;
\item[-] for $1\leq i< j\le k ,\ell\leq n$, $d_{i,\overline k}+d_{j,\ell} = d_{i,\ell} + d_{j,\overline k}$;
\item[-] for $1\leq i< j< k <\ell\leq n-2$, $d_{i,\overline j}+d_{k,\overline \ell} = d_{i,\overline k} + d_{j,\overline \ell} = d_{i,\overline \ell}+d_{j,\overline k}$.
\end{enumerate}

\begin{definition}
The subset $F^\bc$ of $\mathcal{D}(\g)$ is called a \emph{Dynkin cone} of $\mathcal{D}(\g)$.
\end{definition}

The height $\mathrm{ht}(\beta)$ of a positive root $\beta\in\Phi^+$ is by definition the number of simple roots needed to sum up $\beta$. It is clear that the point $\mathbf{h}:=(h_\beta)_{\beta\in\Phi^+}$, $h_\beta:=\mathrm{ht}(\beta)$ is in $F^\bc$, while not in the relative interior in general.

\begin{lemma}
The Dynkin cone is a non-empty polyhedral cone in $\mathcal{D}(\g)$.
\end{lemma}

\subsection{Main results}

Let $\g$ be finite-dimensional complex simple Lie algebra of classical type. For $\g$ being of type $\tt A_n$, $\tt B_n$ and $\tt C_n$, we let $\bc\subseteq [n-1]$ and for $\g$ of type $\tt D_n$, we fix $\bc\subseteq [n-3]$. Elements in $\bc$ are $c_1,\ldots,c_t$ with $1\leq c_1<\ldots<c_t\leq n$.

Let $\wt{\g}$ be the Lie algebra of the same type as $\g$ with rank $n+t$. Fix a triangular decomposition $\wt{\g}=\wt{\n}_+\oplus\wt{\mathfrak{h}}\oplus \wt{\n}_-$ of $\g$ and $\wt{\mathfrak{b}}_+:=\wt{\n}_+\oplus\wt{\mathfrak{h}}$; denote by $\wt{\Phi}$ the set of positive roots in $\wt{\g}$, $\wt{\Phi}^+$ the set of positive roots and $\{\wt{\alpha}_1,\ldots,\wt{\alpha}_{n+t}\}$ the set of simple roots. The weight lattice will be denoted by $\wt{\Lambda}$ and the monoid of dominant integral weights is $\wt{\Lambda}^+$. Fundamental weights are denoted by $\wt{\varpi}_1,\ldots,\wt{\varpi}_{n+t}$. For $\lambda\in\wt{\Lambda}^+$, the simple module of highest weight $\lambda$ is $\wt{V}(\lambda)$. Let $W(\wt{\g})$ be the Weyl group of $\wt{\g}$ with simple reflections $s_1,\ldots,s_{n+t}$.

The connected simply-connected semi-simple algebraic group associated to $\g$ (resp. $\wt{\g}$) will be denoted by $G$ (resp. $\wt{G}$). The corresponding Borel subgroup with Lie algebra $\wt{\mathfrak{b}}_+$ is denoted by $\wt{B}$. For a root $\beta\in\Phi$ (resp. $\wt{\beta}\in\wt{\Phi}$), let $U_\beta$ (resp. $U_{\wt{\beta}}$) denote the root subgroup associated to $\beta$ (resp. $\wt{\beta}$) in $G$ (resp. $\wt{G}$).

Fix $\bd\in\mathrm{relint}(F^\bc)$, we have introduced the graded Lie algebra $\n^\bd_-$, the weight PBW degeneration $V^\bd(\lambda)$ as $U(\n_-^\bd)$-module and the $\bd$-degenerate flag variety 
$\mathcal{F}^\bd(\lambda)$, embedded as a highest weight orbit in $\mathbb{P}(V^\bd(\lambda))$.

The main result of the paper is:

\begin{theorem}\label{Thm:IsoModule}
There exist
\begin{enumerate}
\item[-] a monoid homomorphism $\Psi:\Lambda^+\to\wt{\Lambda}^+$,
\item[-] a Weyl group element $w_\bd\in W(\wt{\g})$,
\item[-] a $U(\n_-^\bd)$-module structure on the Demazure module $\wt{V}_{w_\bd}(\Psi(\lambda))$,
\end{enumerate}
such that $V^\bd(\lambda)\cong\wt{V}_{w_\bd}(\Psi(\lambda))$ as $U(\n_-^\bd)$-modules.
\end{theorem}

As a consequence, we can realise $\mathcal{F}^\bd(\lambda)$ as a Schubert variety. We embed the Schubert variety in the following way: 
$$X_{w_\bd}(\Psi(\lambda))\hookrightarrow \wt{G}/\wt{P}_{\Psi(\lambda)}\hookrightarrow \mathbb{P}(\wt{V}({\Psi(\lambda)}))$$
where $\wt{P}_{\Psi(\lambda)}$ is the parabolic subgroup in $\wt{G}$ stabilising the weight $\Psi(\lambda)$. The image of $X_{w_\bd}(\Psi(\lambda))$ is contained in $\mathbb{P}(\wt{V}_{w_\bd}(\Psi(\lambda)))$.

\begin{corollary}\label{Cor:IsoProj}
For $\bd\in\mathrm{relint}(F^\bc)$, the $\bd$-degenerate flag variety $\mathcal{F}^\bd(\lambda)$ is isomorphic to the Schubert variety $X_{w_\bd}(\Psi(\lambda))$ as projective varieties.
\end{corollary}

From Corollary \ref{Cor:IsoProj}, for any $\bd\in\mathrm{relint}(F^\bc)$, $\mathcal{F}^\bd(\lambda)$ is isomorphic to a Schubert variety, hence it is normal, Cohen-Macaulay and of rational singularities. We do not know whether in general for $\be\in F^\bc\setminus\mathrm{relint}(F^\bc)$, $\mathcal{F}^\be(\lambda)$ is still isomorphic to a Schubert variety. However, we conjecture that certain geometric properties, fulfilled by the Schubert varieties, still hold true for $\mathcal{F}^\be(\lambda)$.

\begin{conjecture}\label{Conj:Geom}
For any $\be\in F^\bc\setminus\mathrm{relint}(F^\bc)$, $\mathcal{F}^\be(\lambda)$ is normal, Cohen-Macaulay and of rational singularities.
\end{conjecture}

\begin{remark}
\begin{enumerate}
\item Since these three geometric properties are open properties, it remains to show that for any $\be\in F^\bc\setminus\mathrm{relint}(F^\bc)$ and $\bd\in\mathrm{relint}(F^\bc)$, there exists a flat family over $\mathbb{C}$ such that the generic fibre is isomorphic to $\mathcal{F}^\bd(\lambda)$ and the special fibre is isomorphic to $\mathcal{F}^\be(\lambda)$. 
\item In type $\tt B_n$, a similar question is asked by Makhlin in \cite[Section 6]{Makh}.
\end{enumerate}
\end{remark}

As a second consequence of the theorem, the isomorphism types of the $\bd$-degenerate flag varieties, as soon as $\bd$ is contained in the relative interior of a Dynkin cone, depend only on the face.

\begin{corollary}
Let $F^\bc$ be a Dynkin cone of $\mathcal{D}(\g)$ and $\bd,\be\in\mathrm{relint}(F)$. Then $\mathcal{F}^{\bd}(\lambda)\cong \mathcal{F}^{\be}(\lambda)$ as projective varieties.
\end{corollary}


\section{Dynkin abelianisation}\label{sec:reduction}

The goal of this section is to reduce the proof of Theorem \ref{Thm:IsoModule} and Corollary \ref{Cor:IsoProj} to the case of fundamental weights. We keep notations as in the previous section. We sometimes drop the upper index $\bd$ from $f_\beta^\bd$ if no confusion occurs from context.

\subsection{A Lie algebra isomorphism}\label{Sec:LAIso}

Fix $\bc=\{c_1,\ldots,c_t\}$ as in previous section. For convenience denote $c_0:=0$. We define a map $\sigma:[n]\to [n+t]$ mapping $1,2,\ldots,n$ to
$$c_0+1,\ldots,c_1,c_1+2,\ldots,c_2+1,c_2+3,\ldots,c_k+k-1,c_k+k+1,\ldots,c_{k+1}+k,\ldots,n+t.$$
Here the missing elements of the image of $\sigma$ in $[n+t]$ are $c_1+1,c_2+2,\ldots,c_t+t$.

We define a map $\psi_\bc:\Phi^+\to\wt{\Phi}^+$ by: for $\beta=k_1\alpha_1+\ldots+k_n\alpha_n\in\Phi^+$ with $1\leq \ell\leq n$ the maximal index with $k_\ell\neq 0$, 
$$\psi_\bc(\beta):=\sum_{i=1}^{\ell-1} k_i (\wt{\alpha}_{\sigma(i)}+\ldots+\wt{\alpha}_{\sigma(i+1)-1})+k_\ell\wt{\alpha}_{\sigma(\ell)}.$$
Note that if $\beta=\alpha_k$ is a simple root, then $\psi_\bc(\alpha_k)=\wt{\alpha}_{\sigma(k)}$.

\begin{remark}\label{Rem:Epsilon}
This formulae simplifies if one denotes the positive roots in terms of the standard basis $\{ \ve_i \}$ as in Section \ref{Subsec:ABCD}:
\begin{enumerate}
\item Type $\tt A_n$: $\psi_\bc(\ve_i-\ve_{j+1})=\ve_{\sigma(i)}-\ve_{\sigma(j)+1}$.
\item Type $\tt B_n$: $\psi_\bc(\ve_i-\ve_{j+1})=\ve_{\sigma(i)}-\ve_{\sigma(j)+1}$; $\psi_\bc(\ve_i)=\ve_{\sigma(i)}$; $\psi_\bc(\ve_i+\ve_j)=\ve_{\sigma(i)}+\ve_{\sigma(j)}$.
\item Type $\tt C_n$: $\psi_\bc(\ve_i-\ve_{j+1})=\ve_{\sigma(i)}-\ve_{\sigma(j)+1}$; $\psi_\bc(2\ve_i)=2\ve_{\sigma(i)}$; $\psi_\bc(\ve_i+\ve_j)=\ve_{\sigma(i)}+\ve_{\sigma(j)}$.
\item Type $\tt D_n$: $\psi_\bc(\ve_i-\ve_{j+1})=\ve_{\sigma(i)}-\ve_{\sigma(j)+1}$; $\psi_\bc(\ve_i+\ve_n)=\ve_{\sigma(i)}+\ve_{\sigma(n)}$; $\psi_\bc(\ve_{n-1}+\ve_n)=\ve_{\sigma(n-1)}+\ve_{\sigma(n)}$; $\psi_\bc(\ve_i+\ve_j)=\ve_{\sigma(i)}+\ve_{\sigma(j)}$. In this case note that $\bc\subseteq[n-3]$ implies that $\sigma(n-2)=n+t-2$, $\sigma(n-1)=n+t-1$ and $\sigma(n)=n+t$.
\end{enumerate}
\end{remark}

The following lemma follows directly from Remark \ref{Rem:Epsilon}.

\begin{lemma}\label{Lem:RootLength}
The map $\psi_\bc$ is well-defined and it preserves the root length.
\end{lemma}

Let $\wt{\Phi}_\bc^+$ denote the image of $\psi_\bc$.

\begin{lemma}\label{Lem:Closed}
Let $\beta_1,\beta_2\in\Phi^+$. Then $\psi_\bc(\beta_1)+\psi_\bc(\beta_2)\in\wt{\Phi}^+$ implies that $\beta_1+\beta_2\in\Phi^+$ and $\psi_\bc(\beta_1)+\psi_\bc(\beta_2)=\psi_\bc(\beta_1+\beta_2)\in\wt{\Phi}_\bc^+$. In particular, the subset $\wt{\Phi}_\bc^+$ is closed in $\wt{\Phi}^+$ in the sense of \cite[Section 26.3]{Hum}.
\end{lemma}

\begin{proof}
This statement holds by the formulae in Remark \ref{Rem:Epsilon}. We prove it in type $\tt B_n$, and all other types can be proved in a similar way.

From Remark \ref{Rem:Epsilon}, it is clear that $\psi_\bc(\beta_1)+\psi_\bc(\beta_2)\in\wt{\Phi}^+$ if and only if one of the following conditions are fulfilled
\begin{enumerate}
\item[-] $\beta_1=\ve_i-\ve_{j+1}$, $\sigma(j)+1=\sigma(j+1)$, $\beta_2=\ve_{j+1}-\ve_k$;
\item[-] $\beta_1=\ve_i-\ve_{j+1}$, $\sigma(j)+1=\sigma(j+1)$, $\beta_2=\ve_{j+1}$;
\item[-] $\beta_1=\ve_i-\ve_{j+1}$, $\sigma(j)+1=\sigma(j+1)$, $\beta_2=\ve_{j+1}+\ve_k$ for $k>j+1$;
\item[-] $\beta_1=\ve_i-\ve_{j+1}$, $\sigma(j)+1=\sigma(j+1)$, $\beta_2=\ve_s+\ve_{j+1}$ for $s<j+1$ with $i\neq s$.
\end{enumerate}
In all these cases $\beta_1+\beta_2\in\Phi^+$ and $\psi_\bc(\beta_1)+\psi_\bc(\beta_2)=\psi_\bc(\beta_1+\beta_2)$.
\end{proof}

We define 
$$\wt{\n}_-^\bc:=\sum_{\wt{\beta}\in\wt{\Phi}^+_\bc}(\wt{\n}_-)_{-\wt{\beta}}$$
as the sum of root subspaces $(\wt{\n}_-)_{-\wt{\beta}}$ of weight $-\wt{\beta}$ where $\wt{\beta}$ runs through $\wt{\Phi}_\bc^+$. According to Lemma \ref{Lem:Closed}, $\wt{\n}_-^\bc$ is a Lie algebra.

\begin{proposition}\label{Prop:LAIso}
For any $\bd\in\mathrm{relint}(F^\bc)$, there exists a Lie algebra isomorphism 
$$\sigma_\bd:\n_-^\bd\cong \wt{\n}_-^\bc.$$
\end{proposition}

\begin{proof}
Recall that in the beginning of the paper we have fixed a Chevalley basis of $\g$. We define $\sigma_\bd$ to be the linear map sending $f_\beta^\bd$ to $f_{\psi_\bc(\beta)}$ for any $\beta\in\Phi^+$. It is clear that $\sigma_\bd$ is an isomorphism of vector spaces.

The study of the Lie bracket will be split into the following three cases:
\begin{enumerate}
\item Let $\alpha,\beta\in\Phi^+$ be such that $\alpha+\beta\notin\Phi^+$. In this case $[f_\alpha^\bd,f_\beta^\bd]=0$. By Lemma \ref{Lem:Closed}, $\psi_\bc(\alpha)+\psi_\bc(\beta)\notin\wt{\Phi}^+$ and hence $[f_{\psi_\bc(\alpha)},f_{\psi_\bc(\beta)}]=0$.
\item When $\alpha+\beta\in\Phi^+$ and $d_\alpha+d_\beta>d_{\alpha+\beta}$, the Lie bracket between the root vectors $f_\alpha^\bd$ and $f_\beta^\bd$ in $\n^\bd_-$ is zero. We show that $\psi_\bc(\alpha)+\psi_\bc(\beta)\notin\wt{\Phi}^+$ hence $[f_{\psi_\bc(\alpha)},f_{\psi_\bc(\beta)}]=0$. 

First note that since $\alpha+\beta\in\Phi^+$, we can assume that $\alpha=\ve_i-\ve_{j+1}$. From the definition of the Dynkin cones $F^\bc$, we have $j\in\bc$. After applying $\psi_\bc$, by Remark \ref{Rem:Epsilon}, $\psi_\bc(\alpha)=\ve_{\sigma(i)}-\ve_{\sigma(j)+1}$. If $\psi_\bc(\alpha)+\psi_\bc(\beta)$ were a positive root, $\psi_\bc(\beta)$ must contain the summand $\ve_{\sigma(j)+1}$ with a positive coefficient, this is not possible since those summands in $\psi_\bc(\beta)$ with positive coefficients are all of form $\ve_{\sigma(k)}$ but it follows from $j\in\bc$ that  $\sigma(j)+1\neq \sigma(j+1)$.

\item It remains to consider the case when $\alpha+\beta\in\Phi^+$ and $d_\alpha+d_\beta=d_{\alpha+\beta}$.

According to the property of the Chevalley basis, if $\beta-r\alpha,\ldots,\beta+q\alpha$ with $q,r\geq 0$ is the $\alpha$-chain through $\beta$, then $[f_\alpha^\bd,f_\beta^\bd]=(r+1)f_{\alpha+\beta}^\bd$, where 
$$r+1=\frac{q(\alpha+\beta,\alpha+\beta)}{(\beta,\beta)}.$$
Similarly the Lie bracket of the two Chevalley basis elements 
$$[f_{\psi_\bc(\alpha)},f_{\psi_\bc(\beta)}]=(r'+1)f_{\psi_\bc(\alpha+\beta)},$$
where $\psi_\bc(\beta)-r'\psi_\bc(\alpha),\ldots,\psi_\bc(\beta)+q'\psi_\bc(\alpha)$ is the $\psi_\bc(\alpha)$-chain through $\psi_\bc(\beta)$ and
$$r'+1=\frac{q'(\psi_\bc(\alpha+\beta),\psi_\bc(\alpha+\beta))}{(\psi_\bc(\beta),\psi_\bc(\beta))}.$$
If $q=q'$ holds, then $r+1=r'+1$ as by Lemma \ref{Lem:RootLength}, $\psi_\bc$ preserves the root lengths. 

To show that $q=q'$ we split into two cases. First notice that: $1\leq q\leq 2$ and $q=2$ (that is to say, $\beta+\alpha$ and $\beta+2\alpha$ are both positive roots) if and only if $\g$ is of type $\tt B_n$, $\beta=\ve_i-\ve_{j+1}$ and $\alpha=\ve_{j+1}$, or $\g$ is or type $\tt C_n$, $\beta=2\ve_{j+1}$ and $\alpha=\ve_i-\ve_{j+1}$. This happens only when $\alpha$ and $\beta$ have different root lengths.

If $q=1$, then $\alpha$ and $\beta$ have the same length, again by Lemma \ref{Lem:RootLength}, $\psi_\bc(\alpha)$ and $\psi_\bc(\beta)$ have the same length, therefore $q'\leq 1$. It remains to show that $\psi_\bc(\alpha)+\psi_\bc(\beta)\in\wt{\Phi}^+$. Since $\alpha+\beta\in\Phi^+$, we can assume that $\alpha=\ve_i-\ve_{j+1}$ and in view of the root length, $\beta=\ve_{j+1}+\ve_k$ for some $k\neq i, j+1$. From $d_\alpha+d_\beta=d_{\alpha+\beta}$ it follows $j\notin\bc$ and hence $\sigma(j)+1=\sigma(j+1)$. Therefore
$$\psi_\bc(\alpha)+\psi_\bc(\beta)=\ve_{\sigma(i)}-\ve_{\sigma(j)+1}+\ve_{\sigma(j+1)}+\ve_{\sigma(k)}=\ve_{\sigma(i)}+\ve_{\sigma(k)}\in\wt{\Phi}^+.$$

If $q=2$, in the two cases above, either $\alpha$ or $\beta$ is of form $\ve_i-\ve_{j+1}$. The same argument as in the case $q=1$ shows that $\sigma(j)+1=\sigma(j+1)$ and hence $\psi_\bc(\beta)+\psi_\bc(\alpha)$ and $\psi_\bc(\beta)+2\psi_\bc(\alpha)$ are in $\wt{\Phi}^+$ and $q'=2$.
\end{enumerate}

As conclusion, we have shown that $\sigma_\bd$ preserves the structure constants, it is therefore an isomorphism.
\end{proof}

\subsection{Monoid morphism}

We set $\Sigma:=\sigma([n])$ to be the image of $\sigma$.

\begin{proposition}\label{Prop:Psi}
There exists a unique monoid homomorphism $\Psi:\Lambda^+\to\wt{\Lambda}^+$ such that for any $\lambda\in\Lambda^+$,
$$(\Psi(\lambda),\psi_\bc(\beta)^\vee)=(\lambda,\beta^\vee) \ \ \text{and} \ \ (\Psi(\lambda),\wt{\alpha}_\ell^\vee)=0$$
for any $\ell\in[n+t]\setminus\Sigma$.
\end{proposition}

\begin{proof}
First notice that the values of $\Psi(\lambda)$ on all simple coroots in $\wt{\g}$ are fixed by the two conditions above. The uniqueness then follows.

For $\lambda\in\Lambda^+$, we define
$$\Psi(\lambda):=\sum_{j=1}^n(\lambda,\alpha_{j}^\vee) \wt{\varpi}_{\sigma(j)}.$$
It follows from the definition that $\Psi$ is a monoid homomorphism. The second equality then holds since the fundamental weights and the simple roots are dual to each other. It remains to verify the first equality. Note that for $\beta=k_1\alpha_1+\ldots+k_n\alpha_n$, the coroot 
$$\beta^\vee=\frac{1}{(\beta,\beta)}\sum_{j=1}^n k_j(\alpha_j,\alpha_j)\alpha_j^\vee.$$
It then follows
\begin{eqnarray*}
(\Psi(\lambda),\psi_\bc(\beta)^\vee) &=& \sum_{j=1}^n (\lambda,\alpha_j^\vee) (\wt{\varpi}_{\sigma(j)},\psi_\bc(\beta)^\vee)\\
&=& \frac{1}{(\psi_\bc(\beta),\psi_\bc(\beta))}\sum_{j=1}^n (\lambda,\alpha_j^\vee)k_{\sigma(j)} (\wt{\alpha}_{\sigma(j)},\wt{\alpha}_{\sigma(j)})\\
&=& \frac{1}{(\beta,\beta)}\left(\lambda,\sum_{j=1}^n k_{j}(\alpha_j,\alpha_j)\alpha_j^\vee\right)\\
&=& (\lambda,\beta^\vee),
\end{eqnarray*}
where we have used the fact that $\psi_\bc$ preserves the root length, and $k_j=k_{\sigma(j)}$ follows from definition.
\end{proof}

\subsection{Weyl group element}

We study the Lie theoretic structure of $\wt{\Phi}_\bc^+$.

Let $w\in W(\wt{\g})$ be a Weyl group element. Denote 
$$\wt{\Phi}_w^+:=\{\wt{\beta}\in\wt{\Phi}^+\mid w(\wt{\beta})<0\}$$ 
the set of positive roots which are sent to negative by $w$. If $w=s_{i_1}\cdots s_{i_r}$ is a reduced decomposition of $w$, then 
\begin{equation}\label{Eq:Phiw}
\wt{\Phi}_w^+=\{\wt{\alpha}_{i_r}, s_{i_r}(\wt{\alpha}_{i_{r-1}}),\ldots,s_{i_r}\cdots s_{i_2}(\wt{\alpha}_{i_{1}})\}.
\end{equation}

We define $w_\bc\in W(\wt{\g})$ to be the Weyl group element
$$w_\bc:=w_{0,t} (s_{t+1}\cdots s_{c_t+t}s_t\cdots s_{c_t+t-1})\cdots (s_3\cdots s_{c_2+2}s_2\cdots s_{c_2+1})(s_2\cdots s_{c_1+1}s_1\cdots s_{c_1})$$
where $w_{0,t}$ is the longest element in the subgroup of $W(\wt{\g})$ generated by $s_{t+1},\ldots,s_{n+t}$. 

\begin{proposition}\label{Prop:WeylGroup}
We have: $\wt{\Phi}_{w_\bc}^+=\wt{\Phi}_\bc^+$.
\end{proposition}

The proof of the proposition will occupy the rest of this subsection. Before going to the proof we discuss a few consequences of the proposition which will be used later.

\begin{corollary}\label{Cor:NoSum}
Let $\wt{\beta}_1,\wt{\beta}_2\in\wt{\Phi}^+$ be two positive roots which are not in $\wt{\Phi}_\bc^+$. Then $\wt{\beta}_1+\wt{\beta}_2\notin\wt{\Phi}_{\bc}^+$.
\end{corollary}

\begin{proof}
If $\wt{\beta}_1+\wt{\beta}_2\in\wt{\Phi}_{\bc}^+$, we need to show that either $\wt{\beta}_1$ or $\wt{\beta}_2$ is in $\wt{\Phi}_{\bc}^+$. This holds by the Theorem on \cite[Page 663]{Papi} because in the language of \emph{loc.cit}, Proposition \ref{Prop:WeylGroup} implies that $\wt{\Phi}_{\bc}^+$ is associated to $w_\bc$ and associated sets form convex sets of roots.
\end{proof}

Let $I(\infty)$ be the left ideal in $U(\wt{\n}_-)$ generated by $f_{\wt{\beta}}$ for $\wt{\beta}\in\wt{\Phi}^+\setminus\wt{\Phi}_{\bc}^+$.

\begin{corollary}\label{Cor:DirectSum}
We have $U(\wt{\n}_-)=U(\wt{\n}_-^\bc)\oplus I(\infty)$ as vector spaces.
\end{corollary}

\begin{proof}
Let $\wt{\beta}_1,\ldots,\wt{\beta}_N$ be an enumeration of positive roots in $\wt{\Phi}^+$ with 
$$\wt{\Phi}_\bc^+=\{\wt{\beta}_1,\ldots,\wt{\beta}_m\}.$$
It then follows from the PBW theorem that $U(\wt{\n}_-)=U(\wt{\n}_-^\bc)+I(\infty)$. To show that the sum is direct, take 
$$x=\sum_{k=m+1}^N x_kf_{\wt{\beta}_k}\in I(\infty)$$
where $x_k\in U(\wt{\n}_-)$ and $f_{\wt{\beta}_k}$ is the fixed root vector of weight $-\wt{\beta}_k$. By Corollary \ref{Cor:NoSum}, for any PBW monomial $f_{\wt{\beta}_1}^{s_1}\cdots f_{\wt{\beta}_N}^{s_N}$ appearing in $x_k$ with a non-zero coefficient, after having written $f_{\wt{\beta}_1}^{s_1}\cdots f_{\wt{\beta}_N}^{s_N}f_{\wt{\beta}_k}$ as a sum of the PBW basis, there is no PBW monomial from $U(\wt{\n}_-^\bc)$ appearing in it with a non-zero coefficient. This shows that for any $k=m+1,\ldots,N$, $x_k=0$ and hence the sum is direct.
\end{proof}

The proof of Proposition \ref{Prop:WeylGroup} is a case-by-case analysis, but the proofs for type $\tt B_n$, $\tt C_n$ and $\tt D_n$ rely on the results in type $\tt A_n$.

\subsubsection{Type $\tt A_n$}
We start from the case where $\g$ is of type $\tt A_n$, then $\wt{\g}$ is of type $\tt A_{n+t}$.

\begin{lemma}\label{Lem:RedDec}
For $1\leq k\leq t+1$, let
$$v_k:=(s_k\cdots s_{c_k+k-1})\cdots (s_k\cdots s_{c_{k-1}+k})$$
where $c_0:=0$ and $c_{t+1}:=n$. Then $w_\bc=v_{t+1}v_t\cdots v_1$.
\end{lemma}

\begin{proof}
We consider the following reduced decomposition of $w_{0,t}$:
$$w_{0,t}^{\tt A_n}=(s_{t+1}\cdots s_{n+t})(s_{t+1}\cdots s_{n+t-1})\cdots (s_{t+1}s_{t+2})s_{t+1}.$$
It follows then
$$w_{0,t}^{\tt A_n} s_{t+1}\cdots s_{c_t+t}=v_t (s_{t+1}\cdots s_{c_t+t-1})\cdots (s_{t+1}s_{t+2})s_{t+1},$$
and hence 
$$w_{0,t}^{\tt A_n} (s_{t+1}\cdots s_{c_t+t}s_t\cdots s_{c_t+t-1})=v_{t+1} (s_{t+1}\cdots s_{c_t+t-1})\cdots (s_{t+1}s_{t+2})s_{t+1}(s_t\cdots s_{c_t+t-1}).$$
Note that on the right hand side, the part without $v_{t+1}$ is a reduced decomposition of the longest element in the Weyl group generated by $s_t,s_{t+1},\ldots,s_{t+c_t-1}$. It means that we can write 
$$w_{0,t}^{\tt A_n} (s_{t+1}\cdots s_{c_t+t}s_t\cdots s_{c_t+t-1})=v_{t+1} (s_t\cdots s_{c_t+t-1})\cdots (s_ts_{t+1})s_t.$$
The above procedure can be then iterated to conclude.
\end{proof}

The word $w_\bc$ coincides with the word appearing in \cite[Section 5.2]{CFFFR17}. It follows from \cite[Proposition 9]{CFFFR17} that $\wt{\Phi}_{w_\bc}^+=\wt{\Phi}_\bc^+$.

For simplicity we set $\ell_j:=\sigma(j)$ for $1\leq j\leq n$ and $\ell_0=0$. If $j\notin \bc$, then $\ell_{j+1}=\ell_j+1$; otherwise $\ell_{j+1}=\ell_j+2$. The permutation given by $w_\bc$ is explicitly described in the following lemma:

\begin{lemma}[{\cite[Proposition 6]{CFFFR17}}]\label{Lem:wcTypeA}
The following explicit formulae for $w_\bc$ hold:
\begin{enumerate}
\item If $\ell_j=\ell_{j-1}+1$, then $w_\bc(\ell_j)=\ell_j+(n-2j+2)$.
\item If $\ell_j=\ell_{j-1}+2$, then $w_\bc(\ell_j-1)=\ell_j-j$ and $w_\bc(\ell_j)=\ell_j+n-j+1$.
\end{enumerate}
\end{lemma}

\subsubsection{Type $\tt BC_n$}\label{Sec:RDBCn}

In this case, we choose the decomposition for $w_{0,t}$ as:
$$w_{0,t}^{\tt BC_n}=s_{n+t}(s_{n+t-1}s_{n+t})\cdots (s_{t+1}\cdots s_{n+t})w_{0,t}^{\tt A_{n-1}},$$
where $w_{0,t}^{\tt A_{n-1}}$ is defined in the proof of Lemma \ref{Lem:RedDec}. The reducedness of the decomposition is not assumed and will be proved as a consequence of Proposition \ref{Prop:WeylGroup}. 

Since $c_t\leq n-1$, $c_t+t\leq n+t-1$, by Lemma \ref{Lem:RedDec}, $w_\bc$ can be written into
\begin{equation}\label{Eq:BCn}
w_\bc=s_{n+t}(s_{n+t-1}s_{n+t})\cdots (s_{t+1}\cdots s_{n+t})v_t\cdots v_1.
\end{equation}
It follows that $\ell(w_\bc)\leq n^2$, where $\ell$ is the length function on the Weyl group.

Since $\#\wt{\Phi}_\bc^+=n^2$, it remains to show that $\wt{\Phi}_\bc^+$ is a subset of $\wt{\Phi}_{w_\bc}^+$, i.e., $w_\bc$ sends all roots in $\wt{\Phi}_{\bc}^+$ to negative roots. 

The Weyl group of type $\tt BC_n$ acts on $\ve_1,\ldots,\ve_n$ as follows: the simple reflections $s_1,\ldots,s_{n-1}$ acts as in the $\tt A_{n-1}$ case; $s_n$ fixes $\ve_i$ for $i\neq n$ and sends $\ve_n$ to $-\ve_n$.

It then follows that 
$$s_{n+t}(s_{n+t-1}s_{n+t})\cdots (s_{t+1}\cdots s_{n+t})$$ 
sends $\ve_{t+1},\ldots, \ve_{n+t}$ to $-\ve_{n+t},\ldots,-\ve_{t+1}$, respectively.

We write again $\ell_j:=\sigma(j)$ for $1\leq j\leq n$ and have a closer look at $w_\bc(\ve_{\sigma(j)})$. If $j=n$, $\ell_n=\sigma(n)=n+t$ hence $v_t\cdots v_1$ will fix $\ve_{\sigma(n)}$ because the maximal index appearing in the simple reflections in $v_t\cdots v_1$ is at most $c_t+t-1\leq n-1+t-1=n+t-2$, hence $w_\bc(\ve_{\ell_n})=-\ve_{t+1}$ is a negative root.

When $1\leq j\leq n-1$, we can apply Lemma \ref{Lem:wcTypeA} for $\tt A_{n-1}$ and obtain that 
$$v_t\cdots v_1(\ve_{\ell_j})=\begin{cases} \ve_{\ell_j+n-2j+1}, & \text{if }\ell_j=\ell_{j-1}+1;\\ \ve_{\ell_j+n-j}, & \text{if }\ell_j=\ell_{j-1}+2,\end{cases}$$
where $\ell_0$ is set to be $0$. In order to show that $w_\bc(\ve_{\ell_j})$ is a negative root, in the first case, it suffices to show that $\ell_j+n-2j+1\geq t+1$. Indeed, notice that $\ell_j+n-2j+1$ is decreasing when $j$ increases from $1$ to $n$, hence its minimum is attained when $\ell_{n}-n+1=t+1$. For the second case, since $\ell_j\geq j+1$, $\ell_j-j+n\geq n+1> t+1$.

Since $w_\bc$ sends all $\ve_{\sigma(j)}$ to negative roots, it maps all $\ve_{\sigma(i)}+\ve_{\sigma(j)}$ to negative roots as well. 

As the word $w_\bc$ has the $\tt A_{n-1}$ word  $w_{0,t}^{\tt A_{n-1}}$ as suffix, by the description of $\wt{\Phi}^+_{w_\bc}$ in \eqref{Eq:Phiw}, the roots of form $\ve_{\sigma(i)}-\ve_{\sigma(j)+1}$ are in $\wt{\Phi}^+_{w_\bc}$.

As conclusion we have shown that in type $\tt B_n$, $\wt{\Phi}^+_\bc$ is a subset of $\wt{\Phi}^+_{w_\bc}$. The same proof applies to type $\tt C_n$, \emph{mutatis mutandis}.

\subsubsection{Type $\tt D_n$}\label{Sec:Dn}

In this case, we choose the decomposition for $w_{0,t}$ in the following way, according to the parity of $n$: when $n$ is even, $w_{0,t}^{\tt D_n}$ is
$$s_{n+t} (s_{n+t-2}s_{n+t-1})\cdots (s_{t+3}\cdots s_{n+t-2}s_{n+t})(s_{t+2}\cdots s_{n+t-1})(s_{t+1}\cdots s_{n+t-2}s_{n+t})w_{0,t}^{\tt A_{n-1}};$$
when $n$ is odd, $w_{0,t}^{\tt D_n}$ has the form
$$s_{n+t-1} (s_{n+t-2}s_{n+t})\cdots (s_{t+3}\cdots s_{n+t-2}s_{n+t})(s_{t+2}\cdots s_{n+t-1})(s_{t+1}\cdots s_{n+t-2}s_{n+t})w_{0,t}^{\tt A_{n-1}},$$
where $w_{0,t}^{\tt A_{n-1}}$ is defined in the proof of Lemma \ref{Lem:RedDec}. Again we do not need the reducedness of the decomposition in the proof.

Again since $c_t\leq n-3$, $c_t+t\leq n+t-3$ and we can apply Lemma \ref{Lem:RedDec} to write $w_{0,t}^{\tt A_{n-1}}$ into $v_t\cdots v_1$. It then follows that $\ell(w_\bc)\leq n(n-1)$ and again we show that $w_\bc$ maps all positive roots in $\wt{\Phi}^+_{\bc}$ to negative roots. The strategy is the same as the $\tt BC_n$ case.

First notice that for the same reason as in the $\tt BC_n$ case, all roots of form $\ve_{\sigma(i)}-\ve_{\sigma(j)+1}$ are in $\wt{\Phi}_{w_\bc}^+$.

When $n$ is even, it is straightforward to check that 
$$s_{n+t} (s_{n+t-2}s_{n+t-1})\cdots(s_{t+2}\cdots s_{n+t-1})(s_{t+1}\cdots s_{n+t-2}s_{n+t})(\ve_{t+k})=-\ve_{n+t-k+1},$$
and when $n$ is odd, the difference is that 
$$s_{n+t-1} (s_{n+t-2}s_{n+t})\cdots (s_{t+2}\cdots s_{n+t-1})(s_{t+1}\cdots s_{n+t-2}s_{n+t})(\ve_{t+k})=-\ve_{n+t-k+1}$$
holds only for $k=2,\ldots,n$, but sends $\ve_{t+1}$ to $\ve_{n+t}$.

If $n$ is even, the same argument as in the type $\tt BC_n$ shows that $w_\bc(\ve_{\sigma(j)})$ are of form $-\ve_{s}$ hence $w_\bc$ sends roots $\ve_{\sigma(i)}+\ve_{\sigma(n)}$ and the roots $\ve_{\sigma(i)}+\ve_{\sigma(j)}$ to negative roots. If $n$ is odd, what will cause problem is when $v_t\cdots v_1(\ve_{\sigma(j)})=\ve_{t+1}$ for some $j$. But from the argument in the type $\tt BC_n$, this will only happen when $j=n$. For the root $\ve_{\sigma(i)}+\ve_{\sigma(n)}$, $w_\bc$ sends $\ve_{\sigma(n)}$ to itself and $\ve_{\sigma(i)}$ to some $-\ve_s$ for $s<\sigma(n)=n+t$. This shows that $\wt{\Phi}_\bc^+$ is a subset of $\wt{\Phi}_{w_\bc}^+$ and the proof of Proposition \ref{Prop:WeylGroup} is then complete.

\begin{corollary}\label{Cor:Reduced}
The above decompositions of $w_\bc$ appeared in Lemma \ref{Lem:RedDec}, in Equation \eqref{Eq:BCn} and for type $\tt D_n$ (with $w_{0,t}^{\tt A_{n-1}}$ replaced by $v_t\cdots v_1$) are reduced.
\end{corollary}

\begin{proof}
It suffices to notice that the length of the decomposition equals to $\#\Phi^+$, which coincides with $\#\wt{\Phi}_\bc^+$ and hence equals to the length of $w_\bc$.
\end{proof}

\subsection{Reduction to fundamental weights}

Let $\bc$ be fixed as in previous subsections and $\bd\in\mathrm{relint}(F^\bc)$. We define $w_\bd:=w_\bc$. According to Lemma \ref{Lem:Face}, the isomorphic type of the Lie algebra $\n_-^\bd$ does not depend on $\bd$ as long as $\bd$ is chosen from $\mathrm{relint}(F^\bc)$. 

\subsubsection{$U(\n_-^\bd)$-module structure}\label{Sec:ModStr}
We start from defining a $U(\n_-^\bd)$-module structure on the Demazure module $\wt{V}_{w_\bc}(\Psi(\lambda))$. We temporarily write $\lambda':=\Psi(\lambda)$. The Demazure module admits the following description
$$\wt{V}_{w_\bd}(\lambda')=U(\wt{\mathfrak{b}}_+)\cdot v_{w_\bd(\lambda')}=w_\bd (w_\bd^{-1}U(\wt{\n}_+) w_\bd)\cdot v_{\lambda'}.$$
Since $v_{\lambda'}$ is a highest weight vector, all positive roots act by zero on it. That is to say, only those root vectors, whose weights are contained in $w_\bd^{-1}\wt{\Phi}^+\cap(-\wt{\Phi}^+)$, may act on $v_{\lambda'}$ in a non-zero way. 

We claim that $w_\bd^{-1}\wt{\Phi}^+\cap(-\wt{\Phi}^+)=-\wt{\Phi}_{w_\bd}^+$. Indeed, if a positive root $\wt{\beta}$ satisfies $w_\bd^{-1}(\wt{\beta})<0$, then $-w_\bd^{-1}(\wt{\beta})$ is a positive root in $\wt{\Phi}_{w_\bd}^+$. Moreover, if a positive root $\wt{\beta}$ satisfies $w_\bd(\wt{\beta})<0$, then $-\beta=w_\bd^{-1}(-w_\bd(\beta))$ is in the set on the left hand side.

By Proposition \ref{Prop:WeylGroup}, we have:
$$\wt{V}_{w_\bd}(\lambda')=w_\bd U(\wt{\n}_-^\bc)\cdot v_{\lambda'}$$
where the Lie algebra $\wt{\n}_-^\bc$ is defined in Section \ref{Sec:LAIso}. Let $\wt{V}_{w_\bd}^-(\lambda'):=U(\wt{\n}_-^\bc)\cdot v_{\lambda'}$. According to Proposition \ref{Prop:LAIso}, there exists a Lie algebra isomorphism $\sigma_\bd:\n_-^\bd\cong \wt{\n}_-^\bc$. We use this isomorphism to define a $U(\n_-^\bd)$-module structure on $\wt{V}_{w_\bd}^-(\lambda')$ by assigning for $x\in\n_-^\bd$, 
$$x\cdot v_{\lambda'}:=\sigma_\bd(x)\cdot v_{\lambda'},$$
and extends to $U(\n_-^\bd)$ by the universal property. Since $\wt{V}_{w_\bd}(\lambda')$ and $\wt{V}_{w_\bd}^-(\lambda')$ are isomorphic as vector spaces, we obtain a $U(\n_-^\bd)$-module structure on $\wt{V}_{w_\bd}(\lambda')$ by pulling back the one on $\wt{V}_{w_\bd}^-(\lambda')$.

\subsubsection{Ansatz}
The goal of this subsection is to reduce the proof of Theorem \ref{Thm:IsoModule} to the following proposition.

\begin{proposition}\label{Prop:Fundamental}
For the monoid homomorphism $\Psi$, the Weyl group element $w_\bd:=w_\bc\in W(\wt{\g})$, and any $1\leq k\leq n$, $V^\bd(\varpi_k)\cong \wt{V}_{w_\bd}(\Psi(\varpi_k))$ as $U(\n_-^\bd)$-modules.
\end{proposition}

The proof of the proposition will occupy the entire Section \ref{Sec:Fund}. Assuming the proposition, we complete the proof of Theorem \ref{Thm:IsoModule}.

\subsubsection{Defining relations}
We start from defining a surjective $U(\n_-^\bd)$-module map 
$$\vp_\lambda^-: \wt{V}^-_{w_\bd}(\Psi(\lambda))\to V^\bd(\lambda).$$
Once it has been defined, composing with the $U(\n_-^\bd)$-module isomorphism $\wt{V}^-_{w_\bd}(\Psi(\lambda))\cong \wt{V}_{w_\bd}(\Psi(\lambda))$ gives a surjective $U(\n_-^\bd)$-module map 
$$\vp_\lambda: \wt{V}_{w_\bd}(\Psi(\lambda))\to V^\bd(\lambda).$$

For simplicity we write again $\lambda':=\Psi(\lambda)$. Note that 
$$\wt{V}^-_{w_\bd}(\lambda')=U(\wt{\n}_-^\bc)\cdot v_{\lambda'} \cong U(\wt{\n}_-^\bc)/I^\bc_{\lambda'}$$ 
where $I^\bc_{\lambda'}:=\mathrm{Ann}_{U(\wt{\n}_-^\bc)}v_{\lambda'}$ is a left ideal in $U(\wt{\n}_-^\bc)$ and 
$$V^\bd(\lambda)=U(\n_-^\bd)\cdot v_\lambda^\bd=U(\n_-^\bd)/\mathrm{Ann}_{U(\n_-^\bd)}v_\lambda^\bd.$$
We denote $J_\lambda:=\sigma^{-1}_\bd(I^\bc_{\lambda'})$, where $\sigma_\bd^{-1}:\wt{\n}_-^\bc\to\n_-^\bd$ is naturally extended to an isomorphism of universal enveloping algebras $\sigma_\bd^{-1}:U(\wt{\n}_-^\bc)\cong U(\n_-^\bd)$.
\begin{proposition}\label{Prop:Ideal}
We have $J_\lambda\subseteq \mathrm{Ann}_{U(\n_-^\bd)}v^\bd_\lambda$.
\end{proposition}

Once this proposition is proved, the isomorphism $\sigma_\bd^{-1}$ will pass to the quotient and yields the desired surjective linear map $\vp_\lambda^-$. Since the $U(\n_-^\bd)$-module structure on $\wt{V}^-_{w_\bd}(\lambda')$ is given by $\sigma_\bd$, such a linear map is a $U(\n_-^\bd)$-module map.

\subsubsection{Proof of Proposition \ref{Prop:Ideal}}

We start from looking at the defining ideal of the Demazure module $\wt{V}_{w_\bd}(\lambda')$. By \cite{Jos}, it is an ideal in $U(\wt{\n}_+)$ generated by 
$$e_{\wt{\gamma}}^{-(w_\bd(\lambda'),\wt{\gamma}^\vee)+1},\text{ for } w_\bd^{-1}(\wt{\gamma})<0\text{ and }e_{\wt{\gamma}},\text{ for }w_\bd^{-1}(\wt{\gamma})>0.$$
For our purpose we twist the ideal to the highest weight vector: if $w_\bd^{-1}(\wt{\gamma})<0$,
$$U(\wt{\n}_+)e_{\wt{\gamma}}^{-(w_\bd(\lambda'),\wt{\gamma}^\vee)+1}v_{w_\bd(\lambda')}=w_\bd U(\wt{\n}_-^\bc)U(\wt{\n}_+^\bc)f_{-w_\bd^{-1}(\wt{\gamma})}^{-(\lambda',w_\bd^{-1}(\wt{\gamma})^\vee)+1}v_{\lambda'},$$
where $\wt{\n}_+^\bc$ is the sum of the root spaces in $\wt{\n}_+$ whose weight $\wt{\beta}$ satisfies $w_\bd(\wt{\beta})>0$. By Proposition \ref{Prop:WeylGroup}, the positive roots appearing as weights in $\wt{\n}_+^\bc$ are those in $\wt{\Phi}^+\setminus \wt{\Phi}_\bc^+$, hence by Corollary \ref{Cor:NoSum}, $\wt{\n}_+^\bc$ is a Lie subalgebra of $\wt{\n}_+$.
Notice that we have shown in Section \ref{Sec:ModStr} that $-\wt{\Phi}_{w_\bc}^+=w_\bd^{-1}\wt{\Phi}^+\cap(-\wt{\Phi}^+)$. Therefore the defining ideal of $\wt{V}_{w_\bd}^-(\lambda')$ in $U(\wt{\n}_-^\bc)$ is, up to intersecting with $U(\wt{\n}_-^\bc)$, generated by:
\begin{equation}\label{Eq:Gen}
U(\wt{\n}_+^\bc)\langle f_{\wt{\beta}}^{(\lambda',\wt{\beta}^\vee)+1}\mid \wt{\beta}\in\wt{\Phi}^+_\bc\rangle.
\end{equation}

Since $v_{\lambda'}$ is a highest weight vector, the action of $\wt{\n}_+^\bc$ on $f_{\wt{\beta}}^k$ is roughly given by derivation. The roots appearing in $\wt{\n}_+^\bc$ are those not in $\wt{\Phi}_\bc^+$, by Corollary \ref{Cor:NoSum}, for $\wt{\alpha}\notin\wt{\Phi}_\bc^+$ and $\wt{\beta}\in\wt{\Phi}^+_\bc$, if $\wt{\beta}-\wt{\alpha}\in\wt{\Phi}^+$, then it is contained in $\wt{\Phi}^+_\bc$. It follows then if  $[e_{\wt{\alpha}},f_{\wt{\beta}}^{(\lambda',\wt{\beta}^\vee)+1}]$ is contained in $U(\wt{\n}_-)$, then it is contained in $U(\wt{\n}_-^\bc)$. In general with the same argument, one shows that for any $f\in U(\wt{\n}_-^\bc)$, if $[e_{\wt{\alpha}},f]$ is contained in $U(\wt{\n}_-)$, then it is contained in $U(\wt{\n}_-^\bc)$.

To complete the proof of the proposition, it remains to show that there exists a subset $W$ of the generating set \eqref{Eq:Gen}, which generates the defining ideal of $\wt{V}_{w_\bd}^-(\lambda')$, such that for any $f\in W$, there exists an element $E\in U(\n_+)$ such that,
\begin{equation}\label{Eq:DiffOp}
\sigma_\bd^{-1}(f)\cdot v_\lambda^\bd=E \sigma_\bd^{-1}(f)\cdot v_\lambda^\bd
\end{equation}
holds. 

To construct $W$, we define for each $\wt\beta\in \wt\Phi_\bc^+$ the subset $S_{\wt\beta}\subset \wt\Phi^+\setminus \wt\Phi_\bc^+$ as follows:
$$S_{\wt\beta} = \{\wt\alpha\in \wt\Phi^+\setminus \wt\Phi_\bc^+ \mid -\wt\alpha + k \wt\beta \in \wt\Phi_\bc^+ \text{ for some } k\geq 1\}.$$
It is useful to note that $\wt\gamma_i\in S_{\wt\beta}\implies \supp\wt\gamma_i\subset \supp \wt\beta$.

The definition of $S_{\wt\beta}$ is so that the action of $e_{\wt\alpha}\in \wt\n^\bc_-$ on $f_{\wt\beta}^{(\lambda',\wt\beta^\vee)+1} \cdot v_{\lambda'}$ is non-zero only if $\wt\alpha\in S_{\wt\beta}$. Therefore we will take for $W$,
$$W = \{e_{\wt\gamma_1}\dots e_{\wt\gamma_m}f_{\wt\beta}^{(\lambda',\wt\beta^\vee)+1} \mid \wt\beta\in \wt\Phi^+_\bc,\ \wt\gamma_i\in S_{\wt\beta}\}\cap U(\wt\n^\bc_-).$$

It remains to verify that elements in $W$ fulfil \eqref{Eq:DiffOp}. We start with a partial section of $\psi_\bc$. 

Let $\pi_\bc:\wt{\Phi}^+\to Q_+:=\mathbb{N}\Phi^+$ be the map defined by:
$$\pi_\bc\left(\sum_{k=1}^{n+t}c_k\wt{\alpha}_k\right)=\sum_{k=1}^n c_{\sigma(k)}\alpha_k.$$
It follows directly from definition that $\pi_\bc\circ \psi_\bc=\mathrm{id}_{\Phi^+}$. Note that the image of $\pi_\bc$ is not necessarily in $\Phi^+$ but in the root monoid $Q_+$. The following lemma makes it precise:

\begin{lemma}\label{Lem:Trio}
For $\wt{\beta}\in\wt{\Phi}^+$, one and only one of the following statements is true:
\begin{enumerate}
\item $\pi_\bc(\wt{\beta})\in\Phi^+$;
\item $\g$ is of type $\tt B_n$, $\wt{\beta}=\wt{\alpha}_{\sigma(i)-1, \overline{\sigma(i)}}$ with $\sigma(i)-1\neq\sigma(i-1)$, then $\pi_\bc(\wt{\beta})=2\alpha_{i,n}\notin\Phi^+$;
\item $\g$ is of type $\tt D_n$, $\wt{\beta}=\wt{\alpha}_{\sigma(i)-1,\overline{\sigma(i)}}$ with $\sigma(i)-1\neq\sigma(i-1)$, then $\pi_\bc(\wt{\beta})=\alpha_{i,n-2}+\alpha_{i,\overline{n-1}}\notin\Phi^+$.
\end{enumerate}
\end{lemma}

\begin{proof}
The Lemma follows from the definition of $\pi_\bc$.
\end{proof}

The following lemma shows that the elements in $W$ (after possibly rearranging the $e_{\wt\gamma_i}$ factors using a PBW argument) fulfil  \eqref{Eq:DiffOp}.

\begin{lemma}\label{Lem:main}
Fix $\wt\beta\in \wt{\Phi}^+$ and suppose $\wt{\gamma}_1,\ldots,\wt{\gamma}_m\in S_{\wt{\beta}}$ and let $e_{\gamma_i}:=e_{\pi_\bc(\wt\gamma_i)}$: here by $e_{\pi_\bc(\wt\alpha_{\sigma(i)-1, \overline{\sigma(i)}})}$ we mean $e^2_{\alpha_{i,n}}$ in type $\tt B$, and $e_{\alpha_{i,n-2}}e_{\alpha_{i,\overline{n-1}}}$ in type $\tt D$.
Then we have 
\begin{equation}
\sigma_\bd^{-1}(e_{\wt{\gamma}_1}\dots e_{\wt\gamma_m}f_{\wt\beta}^k)\cdot v_\lambda^\bd=e_{\gamma_1}\dots e_{\gamma_m} \sigma_\bd^{-1}(f_{\wt\beta}^k)\cdot v_\lambda^\bd
\end{equation}
after possibly rearranging the $e_{\wt\gamma_i}$ and $e_{\gamma_i}$.
\end{lemma}

\begin{proof}
The proof of the Lemma bases on a case by case analysis and will be given in Section \ref{Sec:LemmaProof}.
\end{proof}

\subsubsection{Proof of Theorem \ref{Thm:IsoModule}}

Let $\lambda=\lambda_1\varpi_1+\ldots+\lambda_n\varpi_n$ and consider the following diagram of $U(\n_-^\bd)$-modules and homomorphisms:
\begin{center}\begin{tikzcd}
        V^\bd(\lambda) \arrow[r,two heads,"b"] \arrow[d,two heads,"a"] & (V^\bd(\varpi_1))^{\otimes\lambda_1}\otimes\ldots\otimes (V^\bd(\varpi_n))^{\otimes\lambda_n} \arrow[d,"\rotatebox{90}{\(\sim\)}"] \\
        \wt{V}_{w_\bd}(\Psi(\lambda)) \arrow[r,hook,"d"] & \wt{V}_{w_\bd}(\Psi(\varpi_1))^{\otimes\lambda_1} \otimes\ldots\otimes  \wt{V}_{w_\bd}(\Psi(\varpi_n))^{\otimes\lambda_n}.
    \end{tikzcd}\end{center}

We start from explaining the maps in the diagram: 
\begin{enumerate}
\item[-] $b$ is the projection onto the Cartan component of the tensor product of weighted PBW degenerate modules, the projection is a $U(\n_-^\bd)$-module homomorphism by \cite[Lemma 6.1]{FFL17};
\item[-] the vertical map on the right hand side is the isomorphism given by Proposition \ref{Prop:Fundamental};
\item[-] $d$ is the isomorphism onto the Cartan component of the tensor product of the Demazure modules: it is an isomorphism by the standard monomial theory \cite{LS};
\item[-] $a$ is the composition of $b$, the vertical map on the right hand side and $d^{-1}$ (defined on the Cartan component).
\end{enumerate}
From construction it follows that $a$ is a surjective $U(\n_-^\bd)$-module morphism. According to Proposition \ref{Prop:Ideal}, there exists a surjective $U(\n_-^\bd)$-module morphism 
$$\vp_\lambda:\wt{V}_{w_\bd}(\Psi(\lambda))\to V^\bd(\lambda),$$ 
and hence $a$ is an isomorphism.

\subsubsection{Proof of Corollary \ref{Cor:IsoProj}}

We keep writing $\lambda':=\Psi(\lambda)$ for short. The isomorphism $\vp_\lambda$ induces an isomorphism of projective spaces $\mathbb{P}(\wt{V}_{w_\bd}(\lambda'))\cong\mathbb{P}(V^\bd(\lambda))$.

Consider an open part $\wt{X}^\circ:=\wt{B}\cdot v_{\lambda'}$ of the Schubert variety $X_{w_\bd}(\lambda')$. We show that it is isomorphic to $N_-^{\bd}\cdot v_\lambda^\bd$. The corollary then follows by taking closure on both sides.

If a positive root $\wt{\gamma}\notin\wt{\Phi}_{w_\bd^{-1}}^+$, $e_{\wt{\gamma}}\cdot v_{w_\bd(\lambda')}=0$, hence
$$\wt{B}\cdot v_{\lambda'}=\prod_{\wt{\gamma}\in\wt{\Phi}_{w_\bd^{-1}}^+} U_{\wt{\gamma}}\cdot v_{w_\bd(\lambda')}=w_\bd\prod_{\wt{\beta}\in\wt{\Phi}_{w_\bd}^+}U_{-\wt{\beta}}\cdot v_{\lambda'},$$
where for the last equality we used the identity $w_\bd^{-1}\wt{\Phi}_{w_\bd^{-1}}^+=-\wt{\Phi}_{w_\bd}^+$. Since $\vp_\lambda$ is a $U(\n_-^\bd)$-module isomorphism, keeping in mind the $U(\n_-^\bd)$-module structure on $\wt{V}_{w_\bd}(\lambda')$, we obtain an isomorphism
$$w_\bd\prod_{\wt{\beta}\in\wt{\Phi}_{w_\bd}^+}U_{-\wt{\beta}}\cdot v_{\lambda'}\cong \prod_{\beta\in\Phi^+}U_{-\beta}\cdot v_\lambda^\bd=N^\bd_-\cdot v_\lambda^\bd.$$
The proof of the corollary is then complete.

\section{Proof of Lemma \ref{Lem:main}}\label{Sec:LemmaProof}

\subsection{Type \texorpdfstring{$\tt A$}{A}}

Suppose $\g$ is of type $\tt A_n$, and $\wt\g$ is of type $\tt A_{n+t}$. Without loss of generality (using the fact that $\wt\gamma_i\in S_{\wt\beta}\implies \supp\wt\gamma_i \subset \supp\wt\beta$ and reducing $n$ if necessary) we may assume $f_{\wt\beta} = f_{\wt\alpha_{1,n+t}}$. This yields 
$$\sigma_\bd^{-1}(f^k_{\wt\alpha_{1,n+t}}) = f^k_{\alpha_{1,n}}.$$

The condition $[e_{\wt\gamma_i},\wt f_{\wt{\alpha}_{1,n+t}}]\neq 0$ then implies that each $\wt\gamma_i$ is either of the form $\wt\alpha_{1,\sigma(j)+1}$ or $\wt\alpha_{\sigma(j)+1,n+t}$ for some $j\in \bc = \{c_1,\dots,c_t\}$ defining the Dynkin abelianisation. We may then rearrange these $e_{\wt\gamma_i}$ so that $\wt\gamma_1\preceq\wt\gamma_2\preceq \cdots\preceq \wt\gamma_m$ according to the total order 
    \[\wt\alpha_{1,1} \prec\wt\alpha_{1,2}\prec\ldots \prec\wt\alpha_{1,n+t}\prec\wt\alpha_{2,n+t}\prec\wt{\alpha}_{3,n+t}\prec\ldots\prec\wt\alpha_{n+t,n+t}.\]

Therefore, we can assume that the product of $e_{\wt{\gamma}_i}$ in the statement has the form
\begin{equation}\label{Eq:Product}
\left(\prod^{\uparrow}_{j\in\bc} e^{p_j}_{\wt\alpha_{1,\sigma(j)+1}}\right)\left(\prod^{\uparrow}_{j\in \bc}e^{q_j}_{\wt\alpha_{\sigma(j)+1,n+t}}\right) f_{\wt\alpha_{1,n+t}}^k,
\end{equation}
where $p_j,q_j\in\mathbb{N}$ and the products are executed with respect to the above total order. 

In view of $\pi(\wt\alpha_{\sigma(j)+1, n+t})= \alpha_{j+1, n}$, the claim is then that
\[\sigma_\bd^{-1}\left(\prod^{\uparrow}_{j\in\bc} e^{p_j}_{\wt\alpha_{1,\sigma(j)+1}}
\prod^{\uparrow}_{j\in \bc}e^{q_j}_{\wt\alpha_{\sigma(j)+1,n+t}} f_{\wt\alpha_{1,n+t}}^k\right)\cdot v_\lambda^\bd = \prod^{\uparrow}_{j\in\bc} e^{p_j}_{\alpha_{1,j}}\prod^{\uparrow}_{j\in \bc}e^{q_j}_{\alpha_{j+1}} f_{\alpha_{1,n}}^k\cdot v_\lambda^
\bd.\]

Again, since the monomials on both sides are acted on the highest weight vector $v_\lambda^\bd$, the $e_{\wt{\gamma}}$ will be treated as a differential operator on $f_{\wt\alpha_{1,n+t}}^k$, and from now on we will write $e_{\wt{\gamma}}\cdot f_{\wt\alpha_{1,n+t}}^k$ to emphasise this action without mentioning the action on $v_\lambda^\bd$.
    
From the choice of the Chevalley basis one has 
$$e_{\wt\alpha_{\sigma(j)+1, n+t}}\cdot f_{\wt\alpha_{1, l}}=\begin{cases} 
f_{\wt\alpha_{1, \sigma(j)}}, & \text{if }l = n+t;\\
0, & \text{otherwise.}
\end{cases}$$
Thus if we act only by the second factor in the product in \eqref{Eq:Product} consisting of $e_{\wt\alpha_{\sigma(j)+1,n+t}}$,\ we obtain
\[\prod^{\uparrow}_{j\in \bc}e^{q_j}_{\wt\alpha_{\sigma(j)+1,n+t}}\cdot f_{\wt\alpha_{1,n+t}}^k = \left(\prod^{\downarrow}_{j\in \bc}f^{q_j}_{\wt\alpha_{1, \sigma(j)}}\right)f_{\wt\alpha_{1,n+t}}^r =: \wt{f}\]
with $r:=k-(q_{c_1}+\ldots+q_{c_t})$. Note that in the product $\wt{f}$, we have ordered the terms so that $f_{\wt\alpha_{1,\sigma(j)}}$ appears before $f_{\wt\alpha_{1,\sigma(l)}}$ if $j>l$. We are able to do this because all these root vectors commute in $U(\wt{\g})$ (we may also use a PBW argument to enforce this order, at the cost of splitting the monomial into multiple terms).

With the same argument, we also have
\[\prod^{\uparrow}_{j\in \bc}e^{q_j}_{\alpha_{j+1,n}}\cdot f_{\alpha_{1,n}}^k = \left(\prod^{\downarrow}_{j\in \bc}f^{q_j}_{\alpha_{1, j}}\right)f_{\alpha_{1,n}}^r=:f\in U(\n_-).\]
Hence, all that remains to show is that
\[\sigma_\bd^{-1}\left(\prod^{\uparrow}_{j\in\bc} e^{p_j}_{\wt\alpha_{1,\sigma(j)+1}} \wt{f}\right)\cdot v_\lambda^\bd = \left(\prod^\uparrow_{j\in\bc}e_{\alpha_{1,j}}^{p_j}\right)f\cdot v_\lambda^\bd.\]

We note that on the left side, since we have conveniently ordered the factors in $\wt{f}$, the only non-trivial actions of $e_{\wt\alpha_{1,\sigma(j)+1}}$ are when it acts on $f_{\wt\alpha_{1, \sigma(l)}}$ for $l>j$, in which case $[e_{\wt\alpha_{1,\sigma(j)+1}},f_{\wt\alpha_{1, \sigma(l)}}] = f_{\wt\alpha_{\sigma(j+1), \sigma(l)}}$.

Note that since the differential operators are of form $e_{\wt\alpha_{1,\sigma(j)+1}}$ and root vectors in $\wt{f}$ are of form $f_{\wt\alpha_{1, \sigma(l)}}$, after action there will be no non-zero element in $\mathfrak{h}$ coming out. However this will not be the case if we look at the right hand side.

On the right side, the non-trivial actions of $e_{\alpha_{1,j}}$ are when it acts on $f_{\alpha_{1,l}}$ for $l\geq j$, in which case 
\[
[e_{\alpha_{1,j}},f_{\alpha_{1,l}}] = \begin{cases} f_{\alpha_{j+1,l}}, & \text{if }l>j;\\ h_{\alpha_{1, j}}, & \text{if }l=j. \end{cases}
\]

The similar argument as we deal with the second factor can be applied to show that if we gather all monomials in 
$$\left(\prod^\uparrow_{j\in\bc}e_{\alpha_{1,j}}^{p_j}\right)\cdot f$$
which is in $U(\n_-)$ (this is well-defined by PBW theorem), they correspond via $\sigma_\bd^{-1}$ precisely to the terms on the left.

The final piece of the argument is then to show that: (1). all above-mentioned monomials in $U(\n_-)$ have the same degree; (2). the monomials containing some $h_{\alpha_{1,j}}$ on the right (corresponding to $l=j$) have lower degree than the degree of the monomials in $U(\n_-)$, and hence vanish when they act on $v_\lambda^\bd$.

For the first point, we start by looking at the action of $e_{\alpha_{1,j}}$ on $f$: it acts as a derivation and sends $f_{\alpha_{1,l}}$ for $l>j$ to $f_{\alpha_{j+1,l}}$. If $f$ is a power of a single $f_{\alpha_{1,l}}$ then all monomials obtained from the action of $e_{\alpha_{1,j}}$ on $f$ have the same $\bd$-degree. Otherwise let $1\leq l<l'\leq n$ be such that $f_{\alpha_{1,l}}$ and $f_{\alpha_{1,l'}}$ appears in $f$ with positive exponents. The action of $e_{\alpha_{1,j}}$ on $f_{\alpha_{1,l}}$ decreases the degree by $d_{1,l}-d_{j+1,l}$; and the action of $e_{\alpha_{1,j}}$ on $f_{\alpha_{1,l'}}$ decreases the degree by $d_{1,l'}-d_{j+1,l'}$. From the defining (DO) equality $d_{1,l}+d_{j+1,l'}=d_{1,l'}+d_{j+1,l}$ of the Dynkin cone, they are the same. Since $l$ and $l'$ are arbitrary, we have proved that the action of $e_{\alpha_{1,j}}$ is homogeneous. Iterating this argument shows terminates the proof of the first point.

For the second point, we make use of the (PA) inequalities for $\bd\in\mathrm{relint}(F^\bc)$. Elements in $\mathfrak{h}$ acts as scalars on $v_\lambda^\bd$, hence we set degree $0$ to them. When $l>j$, the action of $e_{\alpha_{1,j}}$ on $f_{\alpha_{1,l}}$ will decrease the degree by $d_{1,l}-d_{j+1,l}$. However, when $e_{\alpha_{1,j}}$ acts on $f_{\alpha_{1,j}}$, the 
degree will decrease by $d_{1,j}$. Since $j\in\bc$, it follows by the definition of the Dynkin cones that
$$d_{1,j}>d_{1,l}-d_{j+1,l},$$
and hence when acted on $v_\lambda^\bd$, the monomials containing certain $h_{\alpha_{1,j}}$ can be neglected.

\subsection{Type \texorpdfstring{$\tt C$}{C}}

The Lie algebra $\g$ is of type $\tt C_n$, so $\wt\g$ is of type $\tt C_{n+t}$. When $\wt\beta = \wt\alpha_{i,j}$, we are in fact in the type $\tt A$ case.

So we assume $\wt\beta = \wt\alpha_{1,\overline{\sigma(j)}}$. Thus $\sigma_\bd^{-1}(f^k_{\wt\alpha_{1,\overline{\sigma(j)}}}) = f^k_{\alpha_{1,\overline j}}$. The condition $\wt\gamma_i\in S_{\wt\beta}$ then implies that each $\wt\gamma_i$ is of the form $\wt\alpha_{1,\sigma(l)+1}$, $\wt\alpha_{\sigma(j),\sigma(l)+1}$ (for $l\ge j$) or $\wt\alpha_{\sigma(l)+1, \overline{\sigma(j)}}$ (for $l< j$) for some $l\in\bc$. We rearrange these $e_{\wt\gamma_i}$ according to the order
\[\wt\alpha_1\prec\cdots \prec \wt\alpha_{n+t} \prec \wt\alpha_{\sigma(j),\sigma(j)+1}\prec \cdots \prec \wt\alpha_{\sigma(j), n+t}\prec \wt\alpha_{1,\overline{\sigma(j)}}\prec \cdots \prec \wt\alpha_{\sigma(j),\overline{\sigma(j)}}.\]

After picking this order for the differential operators, one proceeds just as in the type $\tt A$ case.

\subsection{Type \texorpdfstring{$\tt B$}{B}}

The Lie algebra $\g$ is of type $\tt B_n$, so $\wt\g$ is of type $\tt B_{n+t}$. If $\wt\beta = \wt\alpha_{i,j}$ for $j< n+ t$, then we reduce to the case of type $\tt A$.

If $\wt\beta = \wt\alpha_{i, n+t}$, we may shrink $n$ if necessary and work in a smaller Lie algebra in order to assume $\wt\beta = \wt\alpha_{1,n+t}$. Then the only new $e_{\wt\gamma}$'s to consider are those of the form $\wt\gamma = \wt\alpha_{1, \overline {\sigma(j)+1}}$ for some $j> 1$ in $\bc$. These act on $f^k_{\wt\alpha_{1,n+t}}$ to give a monomial in the root vectors $f_{\wt\alpha_{1,n+t}}$ and $f_{\wt\alpha_{1,\sigma(j)}}$. Thus we may proceed just as in the type $\tt A$ case if we put such $e_{\wt\gamma}$ factors to the rightmost, i.e., declare the order
 \[\wt\alpha_1 \prec\wt\alpha_{1,2}\prec \cdots \prec \wt\alpha_{1,n+t}\prec \wt\alpha_{2,n+t}\prec \cdots\prec \wt\alpha_{n+t} \prec \wt\alpha_{1,\overline{n+t}}\prec \cdots\prec \wt\alpha_{1,\overline 2}.\]

So we assume $\wt\beta = \wt\alpha_{1,\overline{\sigma(j)}}$. Thus $\sigma_\bd^{-1}(f^k_{\wt\alpha_{1,\overline{\sigma(j)}}}) = f^k_{\alpha_{1,\overline j}}$. The condition $\wt\gamma_i\in S_{\wt\beta}$ then implies that each $\wt\gamma_i$ is of the form $\wt\alpha_{1,\sigma(l)+1}$, $\wt\alpha_{\sigma(j),\sigma(l)+1}$ (for $l\ge j$), $\wt\alpha_{\sigma(j),\overline{\sigma(l)+1}}$ (for some $l\ge j$) or $\wt\alpha_{\sigma(l)+1, \overline{\sigma(j)}}$ (for $l< j$) for some $l\in\bc$. We rearrange $e_{\wt\gamma_i}$ according to the order
\begin{multline*}
\wt\alpha_1\prec \cdots \prec \wt\alpha_{n+t} \prec \wt\alpha_{\sigma(j),\sigma(j)+1}\prec \cdots \prec \wt\alpha_{\sigma(j), n+t}\prec\\ \wt\alpha_{\sigma(j), \overline{\sigma(j)+1}}\prec \cdots \prec \wt\alpha_{\sigma(j),\overline{n+t}}\prec\wt\alpha_{1,\overline{\sigma(j)}}\prec\cdots \prec \wt\alpha_{\sigma(j)-1,\overline{\sigma(j)}}.   
\end{multline*}

After picking this order, one proceeds just as in the type $\tt C$ case to see the result. The one main difference from the type $\tt C$ case is that if $\sigma(j)-1\notin \im\sigma$ (equivalently, if $j-1\in \bc$), then $\pi_\bc(\wt\alpha_{\sigma(j)-1,\overline{\sigma(j)}}) = 2\alpha_{j, n}$, which is not a root. Hence we take $e_{\gamma}$ to be $e_{\alpha_{j,n}}^2$. It remains to observe that
\begin{align*}
e_{\wt\alpha_{\sigma(j)-1,\overline{\sigma(j)}}} f^k_{\wt\alpha_{1,\overline{\sigma(j)}}} &= k f^{k-1}_{\wt\alpha_{1,\overline{\sigma(j)}}}f_{1,\sigma(j) - 2}\\
    &= k f^{k-1}_{\wt\alpha_{1,\overline{\sigma(j)}}}f_{1,\sigma(j-1)}
\end{align*}
and
\begin{align*}
    e_{\alpha_{j,n}}^2 f^k_{\alpha_{1,\overline{j}}} &= k e_{\alpha_{j,n}} f^{k-1}_{\alpha_{1,\overline{j}}}f_{1,n}\\
    &= k f^{k-1}_{\alpha_{1,\overline{j}}}f_{1,j-1} + k(k-1) f^{k-2}_{\alpha_{1,\overline{j}}} f^2_{1, n},
\end{align*}
where we have suppressed the structure constants, since we are only dealing with monomials so that scaling is not an issue. A simple calculation of the degrees using the (PA) inequalities shows that the second term here has lower degree, so it must vanish in the associated graded module. More precisely, as we act by all other $e_{\gamma}$, it will always give rise to terms of degree lower than the rest, so we may safely ignore this term in the associated graded module. From the way we chose our order, this is the first $e_\gamma$ we act by, and the rest of the argument is exactly the same as in type $\tt C$.

\subsection{Type \texorpdfstring{$\tt D$}{D}}
$\g$ is of type $\tt D_n$, so $\wt\g$ is of type $\tt D_{n+t}$. If $\wt\beta = \wt\alpha_{i,j}$ for $j\le n+t$, then we reduce to the case of type $\tt A$. The case $\wt\beta = \wt\alpha_{i,\overline{n+t}}$ is also very similar to the type $\tt A$ case, and the exact same arguments work.\\

So we assume $\wt\beta = \wt\alpha_{1,\overline{\sigma(j)}}$. Thus $\sigma_\bd^{-1}(f^k_{\wt\alpha_{1,\overline{\sigma(j)}}}) = f^k_{\alpha_{1,\overline j}}$. The condition $\wt\gamma_i\in S_{\wt\beta}$ then implies that each $\wt\gamma_i$ is of the form $\wt\alpha_{1,\sigma(l)+1}$, $\wt\alpha_{\sigma(j),\sigma(l)+1}$ (for $l\ge j$), $\wt\alpha_{\sigma(j),\overline{\sigma(l)+1}}$ (for some $l\ge j$) or $\wt\alpha_{\sigma(l)+1, \overline{\sigma(j)}}$ (for $l< j$) for some $l\in\bc$. We rearrange $e_{\wt\gamma_i}$ according to the order

\begin{multline*}
\wt\alpha_1\prec\cdots \prec \wt\alpha_{n+t} \prec \wt\alpha_{\sigma(j),\sigma(j)+1}\prec\cdots \prec\wt\alpha_{\sigma(j), n+t}\prec\\ \wt\alpha_{\sigma(j), \overline{\sigma(j)+1}}\prec\cdots \prec \wt\alpha_{\sigma(j),\overline{n+t}}\prec\wt\alpha_{1,\overline{\sigma(j)}}\prec\cdots \prec \wt\alpha_{\sigma(j)-1,\overline{\sigma(j)}}.   
\end{multline*}

After picking this order, one proceed just as in the type $\tt B$ case to see the result.

Similarly to the $\tt B$ case, if $\sigma(j)-1\notin \im\sigma$ (equivalently, if $j-1\in \bc$), then $\pi_\bc(\wt\alpha_{\sigma(j)-1,\overline{\sigma(j)}}) = \alpha_{j, n} + \alpha_{j,n-2}$, which is not a root. Hence we take $e_{\gamma}$ to be $e_{\alpha_{j,n}}e_{\alpha_{j,n-2}}$. It remains to observe that
\begin{align*}
    e_{\wt\alpha_{\sigma(j)-1,\overline{\sigma(j)}}} f^k_{\wt\alpha_{1,\overline{\sigma(j)}}} &= k f^{k-1}_{\wt\alpha_{1,\overline{\sigma(j)}}}f_{1,\sigma(j) - 2}\\
    &= k f^{k-1}_{\wt\alpha_{1,\overline{\sigma(j)}}}f_{1,\sigma(j-1)},
\end{align*}
and
\begin{align*}
    e_{\alpha_{j,n}}e_{\alpha_{j,n-2}} f^k_{\alpha_{1,\overline{j}}} &= k e_{\alpha_{j,n}} f^{k-1}_{\alpha_{1,\overline{j}}}f_{1,n}\\
    &= k f^{k-1}_{\alpha_{1,\overline{j}}}f_{1,j-1} + k(k-1) f^{k-2}_{\alpha_{1,\overline{j}}} f_{1, n} f_{1,n-2},
\end{align*}

where we have once again suppressed the structure constants. A simple calculation of the degrees using (PA) inequalities shows that the second term here has lower degree, so we may safely ignore this term when we take the associated graded module. From the way we chose our order, this is the first $e_\gamma$ we act by, and the rest of the argument is exactly the same as in type $\tt B$.

\section{Dimension counting of Demazure modules}\label{Sec:Fund}

We keep notations as in the previous sections. In this section we prove Proposition \ref{Prop:Fundamental}. According to Proposition \ref{Prop:Ideal}, there exists a surjective map 
$$\vp_{\varpi_k}: \wt{V}_{w_\bd}(\Psi(\varpi_k))\to V^\bd(\varpi_k).$$
We will show that for any $1\leq k\leq n$, 
\begin{equation}\label{Eq:Fund}
\dim \wt{V}_{w_\bd}(\Psi(\varpi_k))\leq\dim V^\bd(\varpi_k).
\end{equation}
Two different methods will be used to establish this inequality, depending on the type of the Lie algebra. For types $\tt A_n$ and $\tt C_n$ we use the fact that $w_\bd$ are triangular elements; and for the orthogonal types $\tt B_n$ and $\tt D_n$, we study in detail the Demazure modules for fundamental weights.

\subsection{Triangular elements: type \texorpdfstring{$\tt A_n$}{A} and \texorpdfstring{$\tt C_n$}{C}}

Triangular elements in the Weyl group of the Lie algebra of type $\tt A_n$ and $\tt C_n$ are introduced in \cite{Fou16,BFK}. For such an element $w$, a basis of the Demazure module $V_w(\lambda)$ is parametrised by the lattice points in a lattice polytope.

Let $\g=\mathfrak{sl}_{n+1}$ be the Lie algebra of type $\tt A_{n}$. For two positive roots $\alpha_{i,j},\alpha_{k,\ell}$,
\begin{enumerate}
\item their \emph{join} is defined to be $\alpha_{i,j}\vee\alpha_{k,\ell}:=\alpha_{\min(i,k),\max(j,\ell)}$;
\item if the intersection of their supports is non-empty, their \emph{meet} is defined to be
$$\alpha_{i,j}\wedge\alpha_{k,\ell}:=\alpha_{\max(i,k),\min(j,\ell)}.$$
Otherwise we say that their meet does not exist.
\end{enumerate}
A subset $A\subset \Phi^+$ is called \emph{triangular}, if 
\begin{enumerate}
\item for any $\alpha,\beta\in A$ such that $\mathrm{supp}(\alpha)\cup\mathrm{supp}(\beta)$ is a consecutive subset of $[n]$, $\alpha\vee\beta\in A$ holds;
\item if the meet $\alpha\wedge \beta$ exists, then it is in $A$.
\end{enumerate} 
An element $w\in W(\g)$ is called \emph{triangular} if $\Phi^+_w$ is a triangular subset.

\begin{lemma}
The element $w_\bc\in W(\wt{\g})$ is triangular.
\end{lemma}

\begin{proof}
According to Proposition \ref{Prop:WeylGroup}, $\wt{\Phi}_{w_\bc}^+=\wt{\Phi}_\bc^+$. It suffices to show that the image of $\psi_\bc$ is a triangular subset in $\wt{\Phi}^+$. This holds because by definition, $\psi_\bc$ preserves the join and the meet operations, and $\Phi^+$ itself is a triangular subset. 
\end{proof}

Let $\g=\mathfrak{sp}_{2n}$ be the Lie algebra of type $\tt C_n$. We embed the Weyl group $W(\g)$ into $W(\mathfrak{sl}_{2n})$, looked as Coxeter groups, by sending the generator $s_i$ to $t_it_{2n-i}$ where $t_1,\ldots,t_{2n-1}$ is the set of generators of $W(\mathfrak{sl}_{2n})$. An element $w\in W(\g)$ is called \emph{triangular} if its image under the above embedding in $W(\mathfrak{sl}_{2n})$ is a triangular element. 

\begin{corollary}
The element $w_\bc\in W(\wt{\g})$ is triangular.
\end{corollary}

\begin{proof}
Note that $w_\bc\in W(\mathfrak{sp}_{2(n+t)})$. We let $w$ be its image in $W(\mathfrak{sl}_{2(n+t)})$ under the above embedding. If $\bc=\{c_1,\ldots,c_t\}$ we denote 
$$\bc':=\{c_1,\ldots,c_t,2n-1-c_t,\ldots,2n-1-c_1\}\subseteq[2n-1].$$ 
The corollary follows from the fact that $w=w_{\bc'}$, where $w_{\bc'}$ is the element in $W(\mathfrak{sl}_{2(n+t)})$ associated to $\bc'$.
\end{proof}

The set $\wt{\Phi}^+$ of positive roots admits a poset structure by requiring: $\wt{\alpha}\succ \wt{\beta}$ if $\wt{\alpha}-\wt{\beta}\in\wt{\Phi}^+$. The advantage of having triangular elements is:

\begin{proposition}[{\cite[Theorem 19]{BFK}}]\label{Prop:Triangular}
For any $\lambda\in\Lambda^+$, a linear basis of $\wt{V}_{w_\bc}(\Psi(\lambda))$ is parametrised by the marked chain polytope of the poset $\wt{\Phi}_\bc^+$ with the poset structure induced from that of $\wt{\Phi}^+$ and the marking is given by $\Psi(\lambda)$.
\end{proposition}

\begin{corollary}
We have $\dim \wt{V}_{w_\bc}(\Psi(\lambda))=\dim V(\lambda)=\dim V^\bd(\lambda)$.
\end{corollary}

\begin{proof}
It follows from the definition that the partial order $\succ$ on the positive roots is preserved by $\psi_\bc$. Therefore Proposition \ref{Prop:Psi} implies that the marked poset in Proposition \ref{Prop:Triangular} coincides with the marked poset on $\Phi^+$ with the marking given by $\lambda$. The corollary is then a consequence of the main results of \cite{FFL11a,FFL11b,ABS}, showing that the lattice points of such a marked poset polytope parametrise a basis of $V(\lambda)$.
\end{proof}

\subsection{Type \texorpdfstring{$\tt A_n$}{A}}
We give another proof of \eqref{Eq:Fund} for type $\tt A_n$. The ideas of the proofs of \eqref{Eq:Fund} in the orthogonal (non-spin) cases are similar and will partially rely on the combinatorics in type $\tt A_n$.

We fix the standard basis $\be_1,\ldots,\be_r$ of $\mathbb{C}^r$. The fundamental representation $V(\varpi_k)$ of $\mathfrak{sl}_{n+1}$ can be realised as the wedge product $\Lambda^k\mathbb{C}^{n+1}$ of dimension $\binom{n+1}{k}$ with $\be_1\wedge\ldots\wedge \be_k$ as a highest weight vector. Again we will write $\ell_i:=\sigma(i)$ for short. Note that $\ell_i-i$ is the number of $c_k$ in $\bc$ with $1\leq c_k<i$.

\begin{lemma}\label{Lem:TypeAExt}
For $1\leq i\leq n$, up to a sign, the extremal weight vector $v_{w_\bc(\varpi_{\ell_i})}$ in $\wt{V}_{w_\bc}(\varpi_{\ell_i})$ is given by:
$$v_{w_\bc(\varpi_{\ell_i})}=\be_1\wedge\ldots\wedge \be_{\ell_i-i}\wedge \be_{n+2+\ell_i-2i}\wedge\ldots\wedge \be_{n+1+\ell_i-i}.$$
\end{lemma}

\begin{proof}
We compute $w_\bc(\{1,2,\ldots,\ell_i\})$. First notice that 
$$\#(\{1,2,\ldots,\ell_i\}\cap \bc)=\ell_i-i.$$
Therefore from Lemma \ref{Lem:wcTypeA} (2), the first statement, $1,\ldots,\ell_i-i$ are in the image. The remaining $s$ elements are determined by Lemma \ref{Lem:wcTypeA} (1) and (2), the second statement: they are $n+2,\ldots,n+1+\ell_i-i$ and $n+1,n,\ldots,n+2+\ell_i-2i$.
\end{proof}

To establish \eqref{Eq:Fund}, notice that the Demazure module 
$$\wt{V}_{w_\bd}(\varpi_{\ell_i})=U(\wt{\n}_+)\cdot v_{w_\bc(\varpi_{\ell_i})}.$$
The action of $U(\wt{\n}_+)$ sends some $\be_p$ to some $\be_q$ with $1\leq q<p$, hence sends a pure wedge product to a pure wedge product. Therefore 
$$\wt{V}_{w_\bd}(\varpi_{\ell_i})\subseteq \mathrm{span}_{\mathbb{C}}\{\be_1\wedge\ldots\wedge \be_{\ell_i-i}\wedge \be_{p_1}\wedge\ldots\wedge\be_{p_i}\mid \ell_i-i+1\leq p_1<\ldots<p_i\leq n+1+\ell_i-i\}.$$
The dimension of the vector space on the right hand side is $\binom{n}{k}$ since the freedoms are the choices of a subset $\{p_1,\ldots,p_i\}$ from the set of $n$ elements $\{\ell_i-i+1,\ldots,n+\ell_i-i+1\}$. This is exactly \eqref{Eq:Fund}.

\subsection{Type \texorpdfstring{$\tt B_n$}{B}: non-spin representations}

For the Lie algebra $\mathfrak{so}_{2n+1}$ of type $\tt B_n$, the non-spin fundamental representations are exterior powers of the vector representation: for $1\leq k\leq n-1$, $V(\varpi_k)=\Lambda^k V(\varpi_1)$ and $V(\varpi_1)\cong\mathbb{C}^{2n+1}$ is the vector representation. We fix a basis $\be_1,\ldots,\be_n,\be_0,\be_{-n},\ldots,\be_{-1}$ of $\mathbb{C}^{2n+1}$ where $\be_{-k}:=\be_{2n+1-k}$ and $\be_0:=\be_{2n+1}$. It carries the following action of $e_{i,j}:=e_{\alpha_{i,j}}$ and $e_{i,\overline{j}}:=e_{\alpha_{i,\overline{j}}}$: up to a non-zero scalar, 
\begin{enumerate}
\item for $1\leq i\leq j\leq n-1$, $e_{i,j}$ sends $\be_{j+1}$ to $\be_i$ and sends $\be_{-i}$ to $\be_{-(j+1)}$;
\item for $1\leq i\leq j\leq n-1$, $e_{i,\overline{j}}$ sends $\be_{-j}$ to $\be_i$ and sends $\be_{-i}$ to $\be_{j}$;
\item for $1\leq i\leq n$, $e_{i,n}$ sends $\be_{0}$ to $\be_i$ and $\be_{-i}$ to $\be_0$.
\end{enumerate}
For $1\leq i\leq n$, the weight of $\be_i$ (resp. $\be_{-i}$) is $\ve_i$ (resp. $-\ve_i$), and the weight of $\be_0$ is $0$.

We will use the following total order on the indices: 
$$1<2<\ldots<n<0<-n<-(n-1)<\ldots<-1.$$
It follows from the description above that the action of $e_{i,j}$, $e_{i,\overline{j}}$ and $e_{i,n}$ decrease the index with respect to this total order when acting on some basis element $\be_i$.

Up to a non-zero scalar, the Weyl group acts on the chosen basis of $V(\varpi_1)$ in the same way as it acts on their weights.

\begin{lemma}\label{Lem:TypeBExt}
For $1\leq i\leq n-1$, up to a non-zero scalar, the extremal weight vector $v_{w_\bc(\varpi_{\ell_i})}$ in $\wt{V}_{w_\bc}(\varpi_{\ell_i})$ is given by:
$$v_{w_\bc(\varpi_{\ell_i})}=\be_1\wedge\ldots\wedge \be_{\ell_i-i}\wedge \be_{-(2t+i+1-\ell_i)}\wedge\ldots\wedge \be_{-(2t+2i-\ell_i)}.$$
\end{lemma}

\begin{proof}
According to \eqref{Eq:BCn}, the Weyl group element $w_\bc$ has the following form:
$$w_\bc=s_{n+t}(s_{n+t-1}s_{n+t})\cdots (s_{t+1}\cdots s_{n+t})v_t\cdots v_1.$$
The same proof as in Lemma \ref{Lem:TypeAExt} shows that
$$v_t\cdots v_1(\be_1\wedge\ldots\wedge \be_{\ell_i})=\be_1\wedge\ldots\wedge\be_{\ell_i-i}\wedge \be_{n+1+\ell_i-2i}\wedge\ldots\wedge \be_{n+\ell_i-i}.$$
Since $s_{n+t}(s_{n+t-1}s_{n+t})\cdots (s_{t+1}\cdots s_{n+t})$ sends $\ve_{t+1},\ldots,\ve_{n+t}$ to $-\ve_{n+t},\ldots,-\ve_{t+1}$, it follows from $\ell_i-i<t+1$ that it sends $n+1+\ell_i-2i$, $\ldots$, $n+\ell_i-i$ to $-(2t+i+1-\ell_i)$, $\ldots$, $-(2t+2i-\ell_i)$. The proof is then complete.
\end{proof}

For the Demazure module $\wt{V}_{w_\bd}(\varpi_{\ell_i})=U(\wt{\n}_+)\cdot v_{w_\bc(\varpi_{\ell_i})}$, it follows from the definition of the action that $\wt{V}_{w_\bd}(\varpi_{\ell_i})$ is contained in 
$$\mathrm{span}_{\mathbb{C}}\{\be_1\wedge\ldots\wedge \be_{\ell_i-i}\wedge \be_{p_1}\wedge\ldots\wedge\be_{p_i}\mid \ell_i-i+1\leq p_1<\ldots<p_i\leq -(2t+2i-\ell_i)\}.$$
There are $2n+1$ elements between $\ell_i-i+1$ and $-(2t+2i-\ell_i)$ in the above total order, and hence 
$$\dim\wt{V}_{w_\bd}(\varpi_{\ell_i})\leq \binom{2n+1}{i}=\dim V(\varpi_i)=\dim V^\bd(\varpi_i).$$

\subsection{Type \texorpdfstring{$\tt D_n$}{D}: non-spin representations}
For the Lie algebra $\mathfrak{so}_{2n}$ of type $\tt D_n$, the non-spin fundamental representations are exterior powers of the vector representation: for $1\leq k\leq n-2$, $V(\varpi_k)=\Lambda^k V(\varpi_1)$ and $V(\varpi_1)\cong\mathbb{C}^{2n}$ is the vector representation. We fix a basis $\be_1,\ldots,\be_n,\be_{-n},\ldots,\be_{-1}$ of $\mathbb{C}^{2n}$ where $\be_{-k}:=\be_{2n-k}$. The total order on the indices of the basis is taken to be
$$1<2<\ldots<n<-n<-(n-1)<\ldots <-1.$$

The proof in type $\tt D_n$ is the same as in type $\tt B_n$. Instead of writing down the full proof, we only point out why it is so. First, for the extremal weight vector, Lemma \ref{Lem:TypeBExt} holds without change. In Section \ref{Sec:Dn}, the action of the Weyl group on $\ve_1,\ldots,\ve_{n+t}$ is almost the same as in the $\tt B_n$ case. The only difference is when $n$ is odd, $\ve_{t+1}$ is sent to $\ve_{n+t}$ but $-\ve_{n+t}$. However in our case, the non-spin fundamental representation $V(\varpi_i)$ satisfies $1\leq i\leq n-2$, hence after having been acted by $v_t\cdots v_1$, $t+1$ does not appear in the set $\{n+1+\ell_i-2i,\ldots,n+\ell_i-i\}$. Indeed, notice that for fixed $i$, the smallest element in the set is $n+1+\ell_i-2i$, and this number decreases with respect to $i$, hence for any $i$, the minimal number which may appear in the above set is $n+1+\ell_{n-2}-2(n-2)$. Since $\bc\subseteq [n-3]$, $\ell_{n-2}=n-2+t$, hence 
$$n+1+\ell_{n-2}-2(n-2)=n+1+n-2+t-2n+4=t+3>t+1.$$

Once this extremal weight vector is described, it suffices to notice that the action of a root vector in $\wt{\n}_+$ will decrease the index of $\be_i$ for $1\leq i\leq -1$ with respect to the above total order. Therefore the same formula on the inclusion of the Demazure module in the vector space generated by pure wedge products still holds without any change. The difference is that to count the dimension of the latter space, since there is no $\be_0$, which sits between $n$ and $-n$ in the $\tt B_n$ case, its dimension is $\binom{2n}{i}$, which coincides with the dimension of $V^\bd(\varpi_i)$.

\subsection{Spin representations}
It remains to deal with the spin representations: for type $\tt B_n$ it is $V(\varpi_n)$ and for type $\tt D_n$, they are $V(\varpi_{n-1})$ and $V(\varpi_n)$. The spin representations are minuscule, and we will adapt the idea in \cite{CLL} to deal with the spin cases. Since in type $\tt D_n$ it is assumed that $\bc\subseteq[n-3]$, it follows that $\sigma(n-1)=n+t-1$ and $\sigma(n)=n+t$.

We will prove that $\dim \wt{V}_{w_\bd}(\varpi_{n+t})=\dim V(\varpi_n)$ for type $\tt B_n$. The proof for type $\tt D_n$ is the same.

We start from the formula \eqref{Eq:BCn} of the reduced decomposition of $w_\bc$ (the reducedness holds by Corollary \ref{Cor:Reduced})
$$w_\bc=s_{n+t}(s_{n+t-1}s_{n+t})\cdots (s_{t+1}\cdots s_{n+t})v_t\cdots v_1.$$
Let $\mathfrak{l}_n$ be the semi-simple part of the Levi subalgebra of $\wt{\g}$ corresponding to the simple roots $\wt{\alpha}_{t+1},\ldots,\wt{\alpha}_{n+t}$. It carries naturally an induced triangular decomposition $\mathfrak{l}_n=\mathfrak{l}_n^+\oplus \mathfrak{l}_n^0\oplus \mathfrak{l}_n^-$ from $\wt{\g}$ where $\mathfrak{l}_n^\pm=\mathfrak{l}_n\cap\wt{\n}_\pm$ and $\mathfrak{l}_n^0=\mathfrak{l}_n\cap\wt{\mathfrak{h}}$. The Lie algebra $\mathfrak{l}_n$ is isomorphic to $\g$: such an isomorphism can be chosen to pair the simple root $\wt{\alpha}_{t+k}$ with $\alpha_k$.

It follows that in the stabiliser $W(\wt{\g})_{\varpi_{n+t}}$ of $\varpi_{n+t}$ in $W(\wt{\g})$, the image of $w_\bc$ is given by
$$\overline{w}_\bc:=s_{n+t}(s_{n+t-1}s_{n+t})\cdots (s_{t+1}\cdots s_{n+t}).$$ 
It is contained in the Weyl group $W(\mathfrak{l}_n)_{\varpi_{n+t}}$ of the Levi subalgebra modulo the stabiliser of $\varpi_{n+t}$, and moreover, it is the longest element therein. Therefore
$$\wt{V}_{w_\bd}(\wt{\varpi}_{n+t})=U(\wt{\n}_+)\cdot v_{w_\bd(\wt{\varpi}_{n+t})}=U(\wt{\n}_+)\cdot v_{\overline{w}_\bd(\wt{\varpi}_{n+t})}=U(\mathfrak{l}_n^+)\cdot v_{\overline{w}_\bd(\wt{\varpi}_{n+t})}.$$
Since $\overline{w}_\bd$ is the longest element, $U(\mathfrak{l}_n^+)\cdot v_{\overline{w}_\bd(\wt{\varpi}_{n+t})}\cong V(\varpi_n)$. The statement on the dimension then follows.

\printbibliography
\end{document}